\documentclass[reqno]{amsart}

\usepackage[a4paper, centering]{geometry}

\usepackage{amssymb,amsmath,latexsym,amsthm,amsfonts}
\usepackage{blindtext}
\usepackage{todonotes}
\usepackage{esint,dsfont}

\def\R{{\mathbb R}}

\def\C{\mathbb{C}}
\def\Z{\mathbb{Z}}

\def\ii{\mathrm{i}}

\newtheorem{prop}{\bf Proposition}[section]
\newtheorem{thm}[prop]{\bf Theorem}
\newtheorem{cor}[prop]{\bf Corollary}
\newtheorem{lem}[prop]{\bf Lemma}
\newtheorem{rmk}[prop]{\it Remark}

\begin{document}

\title{The $p\,$-approximation property for simple Lie groups with finite center}

\author{Ignacio Vergara}
\address{UMPA UMR 5669 CNRS, ENS Lyon, Universit\'e de Lyon, 69364 Lyon Cedex 07, France}
%
%
%
%
%

\begin{abstract}
We prove that, for any $1<p<\infty$, the groups $\text{SL}(3,\mathbb{R})$ and $\text{Sp}(2,\mathbb{R})$ do not have the $p\,$-approximation property of An, Lee and Ruan, which implies in particular that they are not $p\,$-weakly amenable. It follows that the same holds for any connected simple Lie group with finite center and real rank greater than 1, as well as for any lattice in it. This extends Haagerup and de Laat's result for the AP, which in this language corresponds to the case $p=2$.
\end{abstract}

\maketitle

\section{Introduction}

In \cite{An} the authors define what we could call $L_p$ versions of weak amenability and the approximation property (AP) for locally compact groups. In order to do this, they work with the Fig\`a-Talamanca--Herz algebra of a group instead of its Fourier algebra, and consider its $p\,$-completely bounded multipliers. These objects arise from the theory of $p\,$-operator spaces, which finds its origins in Pisier's work \cite{Pis2}. We refer the reader to \cite{An} and \cite{Daw} for a detailed treatment of this subject.

For a complex Banach space $E$, we denote $E^*$ its dual and $\mathcal{B}(E)$ the algebra of bounded operators on $E$. Let $G$ be a locally compact group endowed with a left Haar measure. For $1<p<\infty$ let $\lambda_p:G\to\mathcal{B}(L_p(G))$ be the left regular representation
\begin{equation*}
\lambda_p(s)f(t)=f(s^{-1}t).
\end{equation*}
Consider the Banach space projective tensor product $L_{p'}(G)\hat{\otimes}L_p(G)$, where $p^{-1}+p'^{-1}=1$, and $C_0(G)$ the space of continuous functions on $G$ that vanish at infinity. We may define a map $\Lambda_p:L_{p'}(G)\hat{\otimes}L_p(G)\to C_0(G)$ by
\begin{equation}\label{Lamda_p}
\Lambda_p(g\otimes f)(s)=\langle g,\lambda_p(s)f\rangle=g\ast\check{f}(s),
\end{equation}
where $\check{f}(t)=f(t^{-1})$. The Fig\`a-Talamanca--Herz algebra $A_p(G)$ is defined as the image of $\Lambda_p$ endowed with the norm given by $L_{p'}(G)\hat{\otimes}L_p(G)/Ker(\Lambda_p)$, namely
\begin{equation*}
\| a\|_{A_p(G)}=\inf\left\{\sum_n\|g_n\|_{p'}\|f_n\|_p\ \bigg|\ \begin{array}{c}
a=\sum_ng_n\ast\check{f}_n,\\ \ g_n\in L_{p'}(G),\ f_n\in L_p(G)
\end{array} \right\}.
\end{equation*}
It is a commutative Banach algebra under pointwise operations. Recall that the space $\mathcal{B}(L_p(G))$ may be identified with the dual of $L_{p'}(G)\hat{\otimes}L_p(G)$ by
\begin{equation*}
\langle T,g\otimes f\rangle=\langle g, T(f)\rangle,\quad\forall T\in\mathcal{B}(L_p(G)),\ \forall g\in L_{p'}(G),\ \forall f\in L_p(G).
\end{equation*}
The algebra of $p\,$-pseudo-measures $PM_p(G)$ is defined as the weak*-closed linear span of $\{\lambda_p(s) : s\in G \}$ in $\mathcal{B}(L_p(G))$. This algebra can be identified with $A_p(G)^*$ by
\begin{equation*}
\langle T,\Lambda_p(g\otimes f)\rangle = \langle g, T(f)\rangle,\quad \forall T\in PM_p(G),\ \forall g\in L_{p'}(G),\ \forall f\in L_p(G).
\end{equation*}
When $p=2$, $A_p(G)$ coincides with the Fourier algebra $A(G)$, and $PM_p(G)$ is the group von Neumann algebra $\mathcal{L}(G)$.

We say that a function $\varphi:G\to\C$ is a multiplier of $A_p(G)$ if, for every $a\in A_p(G)$, the function $\varphi a$ defined by pointwise multiplication is again an element of $A_p(G)$. The notion of $p\,$-completely bounded multiplier is defined in the context of $p\,$-operator spaces; however, we will only use the following characterization due to Daws \cite[Theorem 8.3]{Daw}, which can be taken as a definition. Consider $SQ_p$ the class of Banach spaces which are quotients of subspaces (or equivalently subspaces of quotients) of $L_p$ spaces. We say that $\varphi$ is a $p\,$-completely bounded multiplier of $A_p(G)$ if there exist $E\in SQ_p$ and continuous bounded maps $\alpha: G\to E$, $\beta: G\to E^*$ such that
\begin{equation}\label{phiab}
\varphi(st^{-1})=\langle\beta(s),\alpha(t)\rangle,\qquad \forall s,t\in G.
\end{equation}
We denote the space of all such functions $M_{p\text{-}cb}(G)$, with the norm given by the infimum of $\|\alpha\|_{\infty}\|\beta\|_{\infty}$ over all $E$, $\alpha$, $\beta$ such that (\ref{phiab}) holds. Since closed subspaces and quotients of Hilbert spaces are again Hilbert spaces, $M_{2\text{-}cb}(G)$ is the space of completely bounded multipliers of the Fourier algebra, often noted as $M_0A(G)$.

\begin{rmk}\label{SQpLp}
	Suppose that $\varphi\in M_{p\text{-}cb}(G)$ and take $E$, $\alpha$ and $\beta$ such that $\varphi$ decomposes as in (\ref{phiab}). Then there exist a measure space $\Omega$ and a closed subspace $F\subset L_p(\Omega)$ such that $E$ is a subspace of $L_p(\Omega)/F$. Then, for every $t\in G$, $\alpha(t)\in L_p(\Omega)/F$, and this space may be identified with the dual of $F^{\perp}\subset L_q(\Omega)$. By the Hahn-Banach theorem, $\alpha(t)$ extends to an element $\tilde{\alpha}(t)\in L_q(\Omega)^*=L_p(\Omega)$ of same norm. Likewise, for all $s\in G$, $\beta(s)\in E^*$ extends to $\tilde{\beta}(s)\in (L_p(\Omega)/F)^*$, and this space may be identified with $F^{\perp}\subset L_q(\Omega)$. Hence, we get functions $\tilde{\alpha}:G\to L_p(\Omega)$ and $\tilde{\beta}:G\to L_q(\Omega)$ such that $\varphi(st^{-1})=\langle\tilde{\beta}(s),\tilde{\alpha}(t)\rangle$ for all $s,t\in G$. However, we do not know whether these functions are continuous.
\end{rmk}

Denote by $C_b(G)$ the space of bounded continuous functions on $G$. The following inclusions are contractive
\begin{equation*}
A_p(G)\hookrightarrow M_{p\text{-}cb}(G) \hookrightarrow C_b(G).
\end{equation*}
An approximate identity (a.i.) for $A_p(G)$ is a net $(\varphi_i)$ in $A_p(G)$ such that
\begin{equation*}
\|\varphi_ia-a\|_{A_p(G)}\to 0,\quad\forall a\in A_p(G).
\end{equation*}
One of the many characterizations of amenability is the existence of a bounded approximate identity for $A_p(G)$ (for some $1<p<\infty$). We say that $G$ is $p\,$-weakly amenable if it has an approximate identity $(\varphi_i)$ such that $\sup_i\|\varphi_i\|_{M_{p\text{-}cb}(G)}<\infty$. If $p=2$, this is the definition of weak amenability. We may also define an analogue of the Cowling-Haagerup constant.
\begin{equation*}
\Lambda_p(G)=\inf\left\{r\geq 1\ \Big|\ \begin{array}{c}
\text{there is an a.i. } (\varphi_i)\text{ for } A_p(G) \\ \text{ such that } \sup_i\|\varphi_i\|_{M_{p\text{-}cb}(G)}\leq r
\end{array} \right\}.
\end{equation*}
In order to define the $p\,$-approximation property we need to view $M_{p\text{-}cb}(G)$ as a dual space. This can be done as follows. Consider the injective contraction $L_1(G)\to M_{p\text{-}cb}(G)^*$ given by
\begin{equation}\label{alphacb}
\langle f,\varphi\rangle=\int_G f(t)\varphi(t)\,dt,\quad \forall\varphi\in M_{p\text{-}cb}(G),\ \forall f\in L_1(G).
\end{equation}
Let $Q_{p\text{-}cb}(G)$ be the norm closure of $L_1(G)$ in $M_{p\text{-}cb}(G)^*$. Miao proved (see \cite[\S 4]{An}) that
\begin{equation*}
M_{p\text{-}cb}(G)=Q_{p\text{-}cb}(G)^*.
\end{equation*}
We say that $G$ has the $p\,$-approximation property ($p\,$-AP) if there exists a net $(\varphi_i)$ in $A_p(G)\subseteq M_{p\text{-}cb}(G)$ such that $\varphi_i\to 1$ in the weak* topology $\sigma(M_{p\text{-}cb}(G),Q_{p\text{-}cb}(G))$.
\begin{rmk}\label{Apc}
	Define $A_{p,c}(G)=A_p(G)\cap C_c(G)$, where $C_c(G)$ is the space of continuous compactly supported functions on $G$. Then $A_{p,c}(G)$ is dense in $A_p(G)$, and since the inclusion $A_p(G)\hookrightarrow M_{p\text{-}cb}(G)$ is contractive, the net $(\varphi_i)$ can be taken in $A_{p,c}(G)$.
\end{rmk}
For $1<p<\infty$ fixed, amenability implies $p\,$-weak amenability, which in turn implies $p\,$-AP. For a discrete group $\Gamma$, these last two properties have equivalent formulations in terms of approximation properties of the $p\,$-operator spaces $PM_p(\Gamma)$ and $PF_p(\Gamma)$ (see \cite[\S 5]{An} for details).
If $p=2$, the $p\,$-AP is the approximation property (AP) of Haagerup and Kraus \cite{HaaKra}.

One of the motivations for asking whether a group possesses the $p\,$-AP is that this property implies that the convoluters are pseudo-measures. The algebra of $p\,$-convoluters $CV_p(G)$  is defined as the commutant of
\begin{align*}
\{\rho_p(s)\ :\ s\in G\},
\end{align*}
where $\rho_p:G\to\mathcal{B}(L_p(G))$ stands for the right regular representation
\begin{equation*}
\rho_p(s)f(t)=f(ts)\Delta(s)^{\frac{1}{p}},
\end{equation*}
and $\Delta:G\to(0,\infty)$ is the modular function. Since $CV_p(G)$ is weak*-closed and the right and left regular representations commute with each other, it is always true that $PM_p(G)\subseteq CV_p(G)$. Cowling \cite{Cow} showed that when $G$ is weakly amenable, $PM_p(G)=CV_p(G)$ for every $1<p<\infty$, suggesting that similar arguments would show that the same holds when $G$ has the AP. This was explained in detail by Daws and Spronk \cite{DawSp}. With minor changes in their proof, we show that for $1<p<\infty$ fixed, if $G$ has the $p\,$-AP, then $PM_p(G)=CV_p(G)$. In terms of Proposition \ref{q-APp-AP}, this is (a priori) a stronger statement; however, we do not know if there exist groups satisfying the $p\,$-AP for some $p$, but failing to have the AP.

The major part of the paper is devoted to the proof of the following theorem.

\begin{thm}\label{SL3pAP}
	The group $\text{SL}(3,\R)$ does not have the $p\,$-AP for any $1<p<\infty$.
\end{thm}

Next we prove the following result, omitting some details which can be easily filled in with the proof of Theorem \ref{SL3pAP}.

\begin{thm}\label{Sp2pAP}
	The group $\text{Sp}(2,\R)$ does not have the $p\,$-AP for any $1<p<\infty$.
\end{thm}

These two results imply our main theorem.

\begin{thm}\label{LiepAP}
	Let $1<p<\infty$ and $G$ be a connected simple Lie group with finite center.
	\begin{itemize}
		\item[a)] $G$ has the $p\,$-AP if and only if it has real rank 0 or 1.
		\item[b)] If $\Gamma$ is a lattice in $G$, then $\Gamma$ has the $p\,$-AP if and only if $G$ has real rank 0 or 1.
	\end{itemize}
\end{thm}

The second part of this theorem provides examples of exact discrete groups which fail to have the $p\,$-AP for every $1<p<\infty$. One of them is $\text{SL}(3,\Z)$.

Following the work of Lafforgue and de la Salle \cite{LafdlS}, where they proved that $\text{SL}(3,\R)$ does not have the AP, Haagerup and de Laat \cite{HaadL} showed that all connected simple Lie groups with finite center and real rank greater than 1 fail to have the AP. This was achieved by treating the case of $\text{Sp}(2,\R)$ and then using some stability properties of the AP. We follow the same strategy of \cite{HaadL} with some slight differences. Instead of using the general theory of Gelfand pairs and the decomposition of completely bounded multipliers on spherical functions, we work with some families of averaging operators on $L_2(K)$, where $K$ is a maximal compact subgroup of $\text{SL}(3,\R)$ (resp. $\text{Sp}(2,\R)$). This allows, by interpolation, to obtain the same kind of results for $p\,$-completely bounded multipliers. These averaging operators were already considered in \cite{Laf}, \cite{LafdlS} and \cite{Laa}. Also, we don't make use of the Krein-Smulian theorem. Instead of this, we construct a Cauchy net of measures on $\text{SL}(3,\R)$ (resp. $\text{Sp}(2,\R)$) that converges to a linear form that cannot exist if the group has the $p\,$-AP.

\begin{rmk}
	As a continuation of \cite{HaadL}, Haagerup and de Laat \cite{HaadL2} studied the AP for the universal covering group $\widetilde{\text{Sp}}(2,\R)$ of $\text{Sp}(2,\R)$, which allowed them to remove the assumption of finite center. Namely, a connected simple Lie group has the AP if and only if it has real rank 0 or 1. This suggests that the same kind of generalization can be made for Theorem \ref{LiepAP} by treating the case of $\widetilde{\text{Sp}}(2,\R)$.
\end{rmk}

This paper is organized as follows. In section 2 we state some preliminary results regarding the algebra $A_p(G)$ and the $p\,$-AP. We also recall the $KAK$ decomposition of $\text{SL}(3,\R)$ and $\text{Sp}(2,\R)$.
Sections 3 and 4 are devoted to the proofs of Theorems \ref{SL3pAP} and \ref{Sp2pAP} respectively.
In section 5 we prove some stability properties of the $p\,$-AP, which allow us to obtain Theorem \ref{LiepAP} as a consequence of the ones above.
Finally, in section 6 we show how the arguments in \cite{DawSp} can be adapted to prove that the $p\,$-AP implies that $PM_p(G)=CV_p(G)$.

\section{Preliminaries}

In this section we study how the $p\,$-AP relates to the $q\,$-AP for different values of $p$ and $q$. Then we prove that, for a compact group $K$, the inclusion $A_p(K)\hookrightarrow M_{p\text{-}cb}(K)$ is surjective. We also recall the $KAK$ decomposition of a connected semisimple Lie group with finite center. This is an essential step, that will allow us to restrict the analysis to the maximal compact subgroups of $\text{SL}(3,\R)$ and $\text{Sp}(2,\R)$.

\subsection{Relations between $p\,$-AP and $q\,$-AP}

Let $G$ be a locally compact group endowed with a left Haar measure. The following proposition was proven in \cite[\S 6]{An} for $G$ discrete. We include the proof of the general case for the sake of completeness.

\begin{prop}\label{q-APp-AP}
	Let $1<p\leq q\leq 2$ or $2\leq q\leq p<\infty$. If $G$ has the $q\,$-AP, then it has the $p\,$-AP.
\end{prop}
\begin{proof}
	By \cite[Proposition 6.1]{An}, the inclusion $\iota: M_{q\text{-}cb}(G)\hookrightarrow M_{p\text{-}cb}(G)$ is a contraction. Thus, the adjoint map $\iota^*:M_{p\text{-}cb}(G)^*\hookrightarrow M_{q\text{-}cb}(G)^*$ is also a contraction and it maps $Q_{p\text{-}cb}(G)$ to $Q_{q\text{-}cb}(G)$ because
	\begin{equation*}
	|\langle f,\varphi\rangle|\leq \|f\|_{M_{p\text{-}cb}(G)^*}\|\varphi\|_{M_{p\text{-}cb}(G)}\leq \|f\|_{M_{p\text{-}cb}(G)^*}\|\varphi\|_{M_{q\text{-}cb}(G)},
	\end{equation*}
	for all $f\in L_1(G)$ and $\varphi\in M_{q\text{-}cb}(G)$. Since $G$ has the $q\,$-AP, there is a net $(\varphi_i)$ in $A_q(G)$ such that $\varphi_i\to 1$ in $\sigma(M_{q\text{-}cb}(G),Q_{q\text{-}cb}(G))$. Observe that since $C_c(G)\otimes C_c(G)$ is dense in $L_{q'}(G)\hat{\otimes}L_q(G)$ and the map $\Lambda_q:L_{q'}(G)\hat{\otimes}L_q(G)\to A_q(G)$ is contractive, we may take the net in $\Lambda_q(C_c(G)\otimes C_c(G))$, which is equal to $\Lambda_p(C_c(G)\otimes C_c(G))$ as sets, so $\varphi_i=\iota(\varphi_i)\in A_p(G)$. Therefore, for all $\mu\in Q_{p\text{-}cb}$,
	\begin{equation*}
	\langle\iota(\varphi_i),\mu\rangle=\langle\varphi_i,\iota^*(\mu)\rangle\to \langle 1,\iota^*(\mu)\rangle=\langle 1,\mu\rangle.
	\end{equation*}
	Hence, $G$ has the $p\,$-AP.
\end{proof}

Now let $1<p,p'<\infty$ with $p^{-1}+p'^{-1}=1$. For a function $\varphi:G\to\C$ we put $\check{\varphi}(s)=\varphi(s^{-1})$. 

\begin{lem}
	The map $\varphi\mapsto\check{\varphi}$ defines an isometric isomorphism from $M_{p\text{-}cb}(G)$ to $M_{p'\text{-}cb}(G)$.
\end{lem}
\begin{proof}
	Let $\varphi\in M_{p\text{-}cb}(G)$. Let $E\in SQ_p$ and $\alpha:G\to E$, $\beta:G\to E^*$ be continuous bounded functions such that
	\begin{equation*}
	\varphi(st^{-1})=\langle\beta(s),\alpha(t)\rangle,\quad \forall s,t\in G.
	\end{equation*}
	Then $\check{\varphi}(st^{-1})=\langle\beta(t),\alpha(s)\rangle$. Now observe that $E^*\in SQ_{p'}$. Indeed, since $E\in SQ_p$, there exist a measure space $(\Omega,\mu)$ and a closed subspace $F\subset L_p(\mu)$ such that $E$ is a subspace of $L_p(\mu)/F$. So $E^*\cong \left(L_p(\mu)/F\right)^*/E^{\perp}$ and $(L_p(\mu)/F)^*\cong F^{\perp}\subset L_{p'}(\mu)$. Since $\alpha(t)$ defines an element of $E^{**}$ of equal norm, we get that $\check{\varphi}\in M_{p'\text{-}cb}(G)$ and $\|\check{\varphi}\|_{ M_{p'\text{-}cb}(G)}\leq \|\varphi\|_{ M_{p\text{-}cb}(G)}$. The other inequality follows analogously.
\end{proof}

For $\mu\in M_{p'\text{-}cb}(G)^*$, we may define $\check{\mu}\in M_{p\text{-}cb}(G)^*$ by duality
\begin{equation*}
\langle \check{\mu}, \varphi\rangle = \langle\mu, \check{\varphi}\rangle,\quad \forall \varphi\in M_{p\text{-}cb}(G),
\end{equation*}
which again defines an isometric isomorphism between $M_{p'\text{-}cb}(G)^*$ and $M_{p\text{-}cb}(G)^*$.

\begin{prop}\label{pAPiffp'AP}
	If $G$ has the $p\,$-AP, then $G$ has the $p'$-AP.
\end{prop}
\begin{proof}
	Let $(\varphi_i)$ be a net in $A_p(G)$ such that $\varphi_i\to 1$ in $\sigma(M_{p\text{-}cb}(G),Q_{p\text{-}cb}(G))$. The identity $(g\ast\check{f})^{\vee}=f\ast\check{g}$ shows that $\check{\varphi}_i\in A_{p'}(G)$ for all $i$. On the other hand, if $\mu\in Q_{p'\text{-}cb}(G)$ is given by a function $g\in L_1(G)$, a change of variable shows that $\check{\mu}$ is given by the function $\tilde{g}(s)=\Delta(s^{-1})g(s^{-1})$ which satisfies $\|\tilde{g}\|_1=\|g\|_1$. Thus, the isometry $\mu\mapsto\check{\mu}$ maps $Q_{p'\text{-}cb}(G)$ to $Q_{p\text{-}cb}(G)$. Therefore
	\begin{equation*}
	\langle \check{\varphi}_i, \mu\rangle = \langle \varphi_i, \check{\mu}\rangle \to \langle 1, \check{\mu}\rangle=\langle 1, \mu\rangle ,\quad\forall \mu\in Q_{p'\text{-}cb}(G).
	\end{equation*}
\end{proof}

Thanks to Proposition \ref{pAPiffp'AP} we may restrict all the analysis to the case $2\leq p <\infty$.

\subsection{The space $M_{p\text{-}cb}(K)$ for $K$ compact}

Recall that the Fig\`a-Talamanca--Herz algebra $A_p(G)$ is defined as the image of the map $\Lambda_p:L_{p'}(G)\hat{\otimes}L_p(G)\to C_0(G)$ given by (\ref{Lamda_p}). This is a particular case of what Herz calls representative functions (see \cite{Her}). As explained in \cite[\S 8]{Daw} in a more modern language, if $E$ is Banach space and $\pi:G\to\mathcal{B}(E)$ is a strongly-continuous representation such that $\pi(s)$ is an isometry for all $s\in G$, then we can define a map $\Pi:E^*\hat{\otimes}E\to C(G)$ by
\begin{equation*}
\Pi(\mu\otimes x)(s)=\langle \mu,\pi(s)x\rangle.
\end{equation*}
The space $A(\pi)$ is defined as the image of $\Pi$ endowed with the norm of $(E^*\hat{\otimes}E)/Ker(\Pi)$. Thus, by definition, $A_p(G)=A(\lambda_p)$. Moreover, consider $L_p(G;E)$ the Banach space of $E$-valued $L_p$ functions on $G$. The algebraic tensor product $L_p(G)\otimes E$ defines a dense subspace of $L_p(G;E)$ by identifying
\begin{equation*}
f\otimes x \mapsto f(\cdot)x,\quad \forall f\in L_p(G),\ \forall x\in E.
\end{equation*}
See \cite[Chapter 7]{Def} for more details. Then, we may define the representation $\lambda_p\otimes I_E:G\to\mathcal{B}(L_p(G;E))$ by
\begin{equation*}
\left[(\lambda_p\otimes I_E)(s)\right](f\otimes x) = (\lambda_p(s)f)\otimes x,\quad\forall s\in G,\ \forall f\in L_p(G),\ \forall x\in E.
\end{equation*}
Herz proved the following.
\begin{thm}{\cite[Lemma 0]{Her}}\label{AlambdaxI}
	If $E\in SQ_p$, then
	\begin{align*}
		A(\lambda_p\otimes I_E)=A(\lambda_p)=A_p(G).
	\end{align*}
\end{thm}
This theorem allows us to obtain the following characterisation.

\begin{prop}\label{A=McbA}
	Let $1<p<\infty$. If $K$ is a compact group, then $M_{p\text{-}cb}(K)=A_p(K)$ with equality of norms.
\end{prop}
\begin{proof}
	Recall that the inclusion $A_p(K)\hookrightarrow M_{p\text{-}cb}(K)$ is a contraction. Now let $\varphi\in M_{p\text{-}cb}(K)$ and take $E$, $\alpha$ and $\beta$ as in (\ref{phiab}). Since $K$ is compact, we can endow it with its normalized Haar measure. Define $X=L_p(K;E)$ and observe that if $p'$ is the H\"older conjugate of $p$, then every $g\in L_{p'}(K;E^*)$ defines an element of $X^*$ by
	\begin{equation*}
	\langle g,f\rangle_{X^*,X}=\int_K \langle g(s),f(s)\rangle\,ds,\qquad\forall f\in X.
	\end{equation*}
	Thus, for all $s,t\in K$,
	\begin{align*}
	\varphi(s^{-1}t) &= \int_K \varphi(s^{-1}kk^{-1}t)\, dk  \\
	&= \int_K \langle\beta(s^{-1}k),\alpha(t^{-1}k)\rangle\, dk \\
	&= \int_K \langle\beta(k),\alpha(t^{-1}sk)\rangle\, dk\\
	&= \left\langle\beta, \left[(\lambda_p\otimes I_E)(s^{-1}t)\right]\alpha\right\rangle_{X^*,X}.
	\end{align*}
	Therefore, for all $k\in K$, 
	\begin{equation*}
	\varphi(k) = \left\langle\beta, \left[(\lambda_p\otimes I_E)(k)\right]\alpha\right\rangle_{X^*,X}.
	\end{equation*}
	Hence, Theorem \ref{AlambdaxI} implies that $\varphi\in A_p(K)$ with
	\begin{equation*}
	\|\varphi\|_{A_p(K)}\leq\|\beta\|_{L_{p'}(K;E^*)}\|\alpha\|_{L_{p}(K;E)}\leq \sup_{s\in K}\|\beta(s)\| \sup_{t\in K}\|\alpha(t)\|.
	\end{equation*}
	Taking the infimum over all $E$, $\alpha$ and $\beta$, we get $\|\varphi\|_{A_p(K)}\leq\|\varphi\|_{M_{p\text{-}cb}(K)}$.
\end{proof}

The following lemma will be useful when defining new multipliers as averages over a compact subgroup.
\begin{lem}\label{lemcontint}
	Let $K$ be a compact subgroup of $G$ endowed with its normalized Haar measure. Let $E$ be a Banach space and $f:G\to E$ a continuous function. Then the function $\tilde{f}:G\to E$ given by
	\begin{align*}
	\tilde{f}(t)=\int_Kf(k^{-1}t)\,dk
	\end{align*}
	is well defined and continuous.
\end{lem}
\begin{proof}
	This function is well defined because $f$ is continuous and $K$ is compact. For the continuity, we will use the fact that, if $\eta:G\to E$ is continuous and compactly supported, then it is left uniformly continuous (see e.g. the proof of \cite[Proposition 2.6]{Foll}). This means that
	\begin{align*}
	\sup_{t\in G}\left\|\eta(s^{-1}t)-\eta(t)\right\| \xrightarrow[s\to 1]{} 0.
	\end{align*}
	Let $V$ be a compact symmetric neighbourhood of $1\in G$ and define $\tilde{K}=VK$. Observe that this is a compact subset of $G$ and $K\subset\tilde{K}$. Take a function $\psi\in C_c(G)$ such that $\psi=1$ on $\tilde{K}$. Now fix $t\in G$ and define $\eta_t(s)=\psi(s)f(s^{-1}t)$. We know that
	\begin{align*}
	\sup_{k\in K}\left\|\eta_t(s^{-1}k)-\eta_t(k)\right\| \xrightarrow[s\to 1]{} 0.
	\end{align*}
	Moreover, for all $s\in V$ and $k\in K$,
	\begin{align*}
	\eta_t(s^{-1}k) = \psi(s^{-1}k)f(k^{-1}st) = f(k^{-1}st).
	\end{align*}
	Therefore, for all $s\in V$,
	\begin{align*}
	\left\|\tilde{f}(st)-\tilde{f}(t)\right\| &= \left\|\int_K\left(\eta_t(s^{-1}k) - \eta_t(k)\right)\,dk\right\|\\
	& \leq \sup_{k\in K}\left\|\eta_t(s^{-1}k)-\eta_t(k)\right\| \xrightarrow[s\to 1]{} 0.
	\end{align*}
\end{proof}

\subsection{The $KAK$ decomposition}

If $G$ is a connected semisimple Lie group, then its Lie algebra $\mathfrak{g}$ admits a Cartan decomposition $\mathfrak{g}=\mathfrak{k}+\mathfrak{p}$. The real rank of $G$ is defined as the dimension of any maximal abelian subspace $\mathfrak{a}$ of $\mathfrak{p}$. If $G$ has finite center, then it may be decomposed as $G=KAK$, where $K$ is a maximal compact subgroup and $A=\exp\mathfrak{a}$. This decomposition is not unique; however, fixing a Weyl chamber $\mathfrak{a}^+$ and letting $A^+=\exp\mathfrak{a}^+$, we still have $G=K\overline{A^+}K$, where $\overline{A^+}$ is the closure of $A^+$ in $G$. See \cite[\S IX.1]{Hel} for details. The advantage of this decomposition is that for every $g\in G$, if $g=kak'$ with $k,k'\in K$, $a\in\overline{A^+}$, then $a$ is unique. This is the main tool that will allow us to work with averages of multipliers when $G$ is $\text{SL}(3,\R)$ or $\text{Sp}(2,\R)$, since in those cases $K$ and $\overline{A^+}$ are explicit and very well known.

\subsubsection{The $K\overline{A^+}K$ decomposition of $\text{SL}(3,\R)$}
The special linear group $G=\text{SL}(3,\R)$ is the Lie group of $3\times 3$ matrices over $\R$ with determinant 1. The group $K=\text{SO}(3,\R)=\{g\in G\ :\ g^tg=I\}$ is a maximal compact subgroup of $G$. If we put
\begin{equation*}
D(r,s,t)=\left(\begin{matrix} e^r & 0 & 0\\ 0 & e^s & 0\\ 0 & 0 & e^t \end{matrix}\right),\quad r,s,t\in\R,
\end{equation*}
then $\overline{A^+}$ is given by
\begin{equation*}
\overline{A^+}=\{D(r,s,t)\ :\ r\geq s \geq t,\ r+s+t=0\}.
\end{equation*}

\subsubsection{The $K\overline{A^+}K$ decomposition of $\text{Sp}(2,\R)$}
The symplectic group $G=\text{Sp}(2,\R)$ is the Lie group of all $4\times 4$ matrices over $\R$ that preserve the standard symplectic form $\omega(x,y)=\langle Jx, y\rangle$, where
\begin{equation}\label{defJ}
J=\left(\begin{matrix} 0 & 0 & 1 & 0\\ 0 & 0 & 0 & 1\\ -1& 0 & 0 & 0\\ 0 & -1 & 0 & 0 \end{matrix}\right).
\end{equation}
Namely, $G=\{g\in M_4(\R)\ :\ g^tJg=J\}$. A maximal compact subgroup of $G$ is given by
\begin{equation}\label{defK=U2}
K=\left\{\left(\begin{matrix} A & -B\\ B & A\end{matrix}\right)\in M_4(\R) \ \Big|\ A+\ii B\in \text{U}(2)\right\}.
\end{equation}
This group is isomorphic to $\text{U}(2)$ via the following map
\begin{equation}\label{isomKU2}
\left(\begin{matrix} a+\ii b & e+\ii f\\ c+\ii d & g+\ii h\end{matrix}\right)\mapsto
\left(\begin{matrix} a & e & -b & -f\\ c & g & -d & -h\\ b & f & a & e\\ d & h & c & g\end{matrix}\right).
\end{equation}
Let
\begin{equation}\label{Dbetagamma}
D(\beta,\gamma)=\left(\begin{matrix} e^\beta & 0 & 0 & 0\\ 0 & e^\gamma & 0 & 0\\ 0 & 0 & e^{-\beta} & 0\\ 0 & 0 & 0 & e^{-\gamma} \end{matrix} \right),\quad \beta,\gamma\in\R.
\end{equation}
Then $\overline{A^+}=\{D(\beta,\gamma)\ :\ \beta\geq\gamma\geq 0\}$.

\section{The special linear group $\text{SL}(3,\R)$}

This section is devoted to the proof of Theorem \ref{SL3pAP}. In the following we fix $G=\text{SL}(3,\R)$, $K=\text{SO}(3,\R)$ and $p\in[2,\infty)$.

\subsection{The operators $T_\delta$ and averages of multipliers}

The idea of the proof of Theorem \ref{SL3pAP} is to find an element $\mu\in Q_{p\text{-}cb}(G)$ such that $\langle 1,\mu\rangle=1$ and $\langle\varphi,\mu\rangle=0$ for every $\varphi\in M_{p\text{-}cb}(G)\cap C_c(G)$. In order to do this, we will construct a Cauchy net of functions $(f_i)$ in $L_1(G)$ such that
\begin{equation*}
\int f_i = 1\quad\text{and}\quad \int f_i\varphi \to 0,
\end{equation*}
for every $\varphi\in C_c(G)$. The main tool used for this construction is the family of operators $(T_\delta)$ considered in \cite[\S 2]{Laf} and \cite[\S 5]{LafdlS}, which can be viewed as operators on $L_p(K)$ when identifying the sphere $S^2$ with a quotient of $\text{SO}(3,\R)$. We detail this now. Let $U$ be the subgroup of $K$ given by
\begin{equation*}
U=\left(\begin{matrix}
1 & 0\\
0 & \text{SO}(2,\R)
\end{matrix}\right).
\end{equation*}
For $\delta\in[-1,1]$, we define $T_{\delta}\in\mathcal{B}(L_p(K))$ by
\begin{equation}\label{defT_d}
T_{\delta}f=\int_U\int_U\lambda_p(uk_{\delta}u')f\,du\,du',\quad\forall f\in L_p(G),
\end{equation}
where
\begin{equation*}
k_{\delta}=\left(\begin{matrix}
\delta & -\sqrt{1-\delta^2} & 0\\
\sqrt{1-\delta^2} & \delta & 0\\
0 & 0 & 1
\end{matrix}\right) \in K.
\end{equation*}
This operator is well defined because, for every $f\in L_p(G)$, the map
\begin{align*}
s\in G\mapsto\lambda_p(s)f \in L_p(G)
\end{align*}
is continuous. Observe that $U\backslash K$ may be identified with the sphere $S^2$ by
\begin{equation*}
Uk\mapsto k^{-1}e_1,
\end{equation*}
where $e_1=\left(\begin{matrix}1\\ 0\\ 0\end{matrix}\right)$. If we put
\begin{equation*}
{}^U\!L_2(K)=\{f\in L_2(K)\ :\ \forall u\in U,\ \lambda_2(u)f=f \},
\end{equation*}
then this identification defines an isometry $\Psi:L_2(S^2)\to {}^U\!L_2(K)$ by
\begin{equation*}
[\Psi f](k)=f(k^{-1}e_1),\quad \forall f\in C(S^2),\ \forall k\in K,
\end{equation*}
which extends to $L_2(S^2)$ by density. This isometry relates the operators $T_\delta$ with those defined in \cite{Laf}. Since $T_\delta(L_2(K))\subseteq {}^U\!L_2(K)$, the operator $\Theta_\delta=\Psi^{-1}T_\delta \Psi$ is well defined on $L_2(S^2)$.

\begin{prop}\label{JTdJ}
	For all $f\in C(S^2)$, $x\in S^2$ and $\delta\in(-1,1)$, $[\Theta_\delta f](x)$ is the average of $f$ on the circle $\{y\in S^2 : \langle x,y\rangle=\delta\}$. Moreover, $[\Theta_1 f](x)=f(x)$ and $[\Theta_{-1} f](x)=f(-x)$.
\end{prop}
\begin{proof}
	First observe that all the vectors $y\in S^2$ that satisfy $\langle e_1,y\rangle=\delta$ are of the form
	\begin{equation*}
	\left(\begin{matrix} \delta\\ \sqrt{1-\delta^2}\cos\theta\\  \sqrt{1-\delta^2}\sin\theta\end{matrix}\right),
	\end{equation*}
	with $\theta\in [0,2\pi)$. So if we put $x=k^{-1}e_1$ with $k\in K$, then
	\begin{align*}
	\langle x,y\rangle=\delta 
	& \iff  \exists\theta\in [0,2\pi), \quad y=k^{-1}\left(\begin{matrix} \delta\\ \sqrt{1-\delta^2}\cos\theta\\  \sqrt{1-\delta^2}\sin\theta\end{matrix}\right)\\
	& \iff  \exists\theta\in [0,2\pi), \quad y=k^{-1}u(\theta)k_\delta e_1,
	\end{align*}
	where
	\begin{equation*}
	u(\theta)=\left(\begin{matrix} 1 & 0 & 0\\ 0 & \cos\theta & -\sin\theta\\ 0 & \sin\theta & \cos\theta\end{matrix}\right),\quad \theta\in [0,2\pi).
	\end{equation*}
	Thus
	\begin{align*}
	[\Theta_\delta f](x) 
	&= [T_\delta \Psi f] (k) \\
	&= \int_U \lambda_p(u k_\delta)[\Psi f](k)\, du\\
	&= \int_U [\Psi f](k_\delta^{-1}u^{-1}k)\, du\\
	& =  \int_U f(k^{-1}uk_\delta e_1)\, du\\
	&= \frac{1}{2\pi}\int_0^{2\pi} f(k^{-1}u(\theta)k_\delta e_1)\, d\theta.
	\end{align*}
	So if $\delta\in(-1,1)$,
	\begin{equation*}
	[\Theta_\delta f] (x)=\frac{1}{2\pi\sqrt{1-\delta^2}}\int\limits_{\langle x,y\rangle=\delta} f(y)\, dy = \fint\limits_{\langle x,y\rangle=\delta} f(y) \, dy.
	\end{equation*}
	Finally, $[\Theta_1 f](x)=f(k^{-1}e_1)=f(x)$ and $[\Theta_{-1} f] (x)=f(-k^{-1}e_1)=f(-x)$.
\end{proof}

The previous proposition implies that the operators $\Theta_\delta$ are exactly those defined in \cite[\S 2]{Laf}. The following estimate corresponds to \cite[Lemme 2.2.a]{Laf} together with \cite[Lemma 3.11]{HaadL}.

\begin{lem}\label{estimThetad}
	For all $\delta\in[-1,1]$,
	\begin{equation*}
	\|\Theta_{\delta}-\Theta_0\|_{\mathcal{B}(L_2)} \leq  4|\delta|^{\frac{1}{2}}.
	\end{equation*}
\end{lem}

This inequality allows us to obtain the following estimate by interpolation.

\begin{lem}\label{estimTd}
	For all $\delta\in[-1,1]$,
	\begin{equation}\label{Td1}
	\|T_{\delta}-T_0\|_{L_p\to L_p} \leq  2^{1+\frac{2}{p}}|\delta|^{\frac{1}{p}}.
	\end{equation}
\end{lem}
\begin{proof}
	Observe that, for all $f\in L_2(K)$, there exists $\tilde{f}\in {}^U\!L_2(K)$ such that $T_\delta\tilde{f}=T_\delta f$ and $\|\tilde{f}\|_2\leq\|f\|_2$, namely
	\begin{equation*}
	\tilde{f}=\int_U\lambda_2(u)f\, du.
	\end{equation*}
	So, by Lemma \ref{estimThetad}, for all $f\in L_2(K)$,
	\begin{align*}
	\|T_\delta f - T_0 f\|_2 
	&= \|T_\delta\tilde{f} - T_0\tilde{f}\|_2\\
	&= \|\Psi(\Theta_\delta-\Theta_0)\Psi^{-1}\tilde{f}\|_2\\
	& \leq \|\Theta_\delta-\Theta_0\|_{\mathcal{B}(L_2)}\|\tilde{f}\|_2\\
	& \leq 4|\delta|^{\frac{1}{2}} \|f\|_2.
	\end{align*}
	This proves the result when $p=2$. If $p>2$, choose $\theta$ and $q$ such that $1-\frac{2}{p}<\theta<1$ and $\frac{1}{p}=\frac{1-\theta}{2}+\frac{\theta}{q}$. This implies that $q>p$. Observe that
	\begin{equation*}
	\|T_{\delta}-T_0\|_{L_q\to L_q}\leq \|T_{\delta}\|_{L_q\to L_q}+\|T_0\|_{L_q\to L_q}\leq 2.
	\end{equation*}
	By interpolation we get
	\begin{equation*}
	\|T_{\delta}-T_0\|_{L_p\to L_p}\leq (4|\delta|^{\frac{1}{2}})^{1-\theta}2^{\theta}.
	\end{equation*}
	Taking $\theta\to 1-\frac{2}{p}$, we obtain (\ref{Td1}).
\end{proof}

\begin{rmk}
	In the previous proof, we don't take $q=\infty$ directly because in \eqref{defT_d}, we defined $T_\delta$ only for $1<p<\infty$, in which case the left regular representation is continuous for the strong operator topology. To avoid giving a precise definition of $T_\delta$ for $p=\infty$ and since this case is not relevant for the present article, we make use of the limit process $q\to\infty$.
\end{rmk}

For $\delta\in[-1,1]$, define $\varepsilon(\delta)=2^{1+\frac{2}{p}}|\delta|^{\frac{1}{p}}$. So the previous lemma states that
\begin{equation*}
\|T_\delta-T_0\|_{\mathcal{B}(L_p(K))}\leq\varepsilon(\delta),\qquad \forall\delta\in[-1,1].
\end{equation*}
Observe that, since the left regular representation commutes with the right regular representation, the same holds for the operators $T_\delta$, namely,
\begin{equation*}
T_\delta\rho_p(k)=\rho_p(k) T_\delta,\qquad \forall k\in K.
\end{equation*}
So $T_\delta\in CV_p(K)$. Furthermore, $K$ is compact, so in particular it is amenable, and by \cite[Theorem 5]{Her2}, $CV_p(K)=PM_p(K)$. Hence $T_\delta\in PM_p(K)$ for all $\delta\in[-1,1]$.

Now we state the main inequality used in the proof of Theorem \ref{SL3pAP}, which involves averages of multipliers. For a continuous function $\varphi:G\to\C$, we define
\begin{equation}\label{avphi}
\tilde{\varphi}(t)=\int_K\int_K\varphi(ktk')\,dk\,dk', \qquad \forall t\in G.
\end{equation}
This function is continuous and $K$-biinvariant. We will also consider the following family of elements of $G$,
\begin{equation*}
D_{a}=\left(\begin{matrix}
e^{a} & 0 & 0\\
0 & e^{-\frac{a}{2}} & 0\\
0 & 0 & e^{-\frac{a}{2}}
\end{matrix}\right),\qquad\forall a\in\R.
\end{equation*}
These matrices satisfy $uD_{a}=D_{a}u$ for every $u\in U$.

\begin{lem}\label{phiDkD}
	For every $\varphi\in M_{p\text{-}cb}(G)$, $a\in\R$ and $\delta\in[-1,1]$,
	\begin{equation}\label{phitil}
	|\tilde{\varphi}(D_{a}k_{\delta}D_{a})-\tilde{\varphi}(D_{a}k_0D_{a})|\leq \varepsilon(\delta)\|\varphi\|_{M_{p\text{-}cb}(G)}.
	\end{equation}
\end{lem}
\begin{proof}
	Consider the function $\phi:K\to\C$ given by
	\begin{equation*}
	\phi(s)=\tilde{\varphi}(D_{a}sD_{a})=\int_K\int_K\varphi(kD_{a}sD_{a}k')\,dk\,dk'.
	\end{equation*}
	Since $\varphi\in M_{p\text{-}cb}(G)$, for all $s,t\in K$,
	\begin{align*}
	\phi(st^{-1}) &= \int_K\int_K\varphi(kD_{a}s(k'^{-1}D_{a}^{-1}t)^{-1})\,dk\,dk' \\
	&= \int_K\int_K\langle\beta(kD_{a}s),\alpha(k'^{-1}D_{a}^{-1}t)\rangle\,dk\,dk'\\
	&= \left\langle\int_K\beta(kD_{a}s)\,dk,\int_K\alpha(k^{-1}D_{a}^{-1}t)\,dk\right\rangle,
	\end{align*}
	where $\alpha$ and $\beta$ are as in (\ref{phiab}). The integrals in the last expression are well defined because $\alpha$ and $\beta$ are continuous and $K$ is compact. So we have found functions $\tilde{\alpha}:K\to E$ and $\tilde{\beta}:K\to E^*$ given by
	\begin{equation*}
	\tilde{\alpha}(t)=\int_K\alpha(k^{-1}D_{a}^{-1}t)\,dk,\qquad\tilde{\beta}(t)=\int_K\beta(kD_{a}t)\,dk,\qquad\forall t\in G,
	\end{equation*}
	such that
	\begin{equation*}
	\phi(st^{-1})=\langle\tilde{\beta}(s),\tilde{\alpha}(t)\rangle
	\end{equation*}
	and
	\begin{equation*}
	\sup_{s\in K}\|\tilde{\beta}(s)\|\leq \sup_{s\in G}\|\beta(s)\|,\qquad \sup_{t\in K}\|\tilde{\alpha}(t)\|\leq\sup_{t\in G}\|\alpha(t)\|.
	\end{equation*}
	Finally, $\tilde{\alpha}$ and $\tilde{\beta}$ are continuous thanks to 
	Lemma \ref{lemcontint}. We conclude that $\phi\in M_{p\text{-}cb}(K)$ and $\|\phi\|_{M_{p\text{-}cb}(K)}\leq\|\varphi\|_{M_{p\text{-}cb}(G)}$. Hence, by Proposition \ref{A=McbA}, $\phi\in A_p(K)$ and $\|\phi\|_{A_p(K)}=\|\phi\|_{M_{p\text{-}cb}(K)}$. Recall that $T_\delta\in PM_p(K)=A_p(K)^*$, thus, using Lemma \ref{estimTd},
	\begin{equation*}
	|\langle T_\delta - T_0, \phi\rangle |\leq \|T_\delta - T_0\|_{\mathcal{B}(L_p(K))} \|\phi\|_{A_p(K)} \leq \varepsilon(\delta) \|\varphi\|_{M_{p\text{-}cb}(G)}.
	\end{equation*}
	On the other hand,
	\begin{align*}
	\langle T_\delta, \phi\rangle &= \int_U\int_U\langle\lambda_p(uk_\delta u'),\phi\rangle\,du\,du' \\
	&= \int_U\int_U\phi(uk_\delta u')\,du\,du'\\
	&= \int_U\int_U\int_K\int_K\varphi(kD_{a}uk_\delta u'D_{a}k')\,dk\,dk'\,du\,du'\\
	&= \int_U\int_U\int_K\int_K\varphi(kuD_{a}k_\delta D_{a}u'k')\,dk\,dk'\,du\,du'\\
	&= \int_U\int_U\int_K\int_K\varphi(kD_{a}k_\delta D_{a}k')\,dk\,dk'\,du\,du'\\
	&= \int_K\int_K\varphi(kD_{a}k_\delta D_{a}k')\,dk\,dk' \\
	&= \tilde{\varphi}(D_{a}k_\delta D_{a}).
	\end{align*}
	And so we obtain (\ref{phitil}).
\end{proof}

\begin{rmk}
	As observed by one of the referees, the use of Proposition \ref{A=McbA} is not necessary for proving Lemma \ref{phiDkD}. Indeed, by \cite[Lemma 8.2]{Daw}, there is a bounded linear map $M_\phi : PM_p(K)\to PM_p(K)$ such that $\|M_\phi\|\leq\|\phi\|_{M_{p\text{-}cb}(K)}$ and $M_\phi\lambda_p(s)=\phi(s)\lambda_p(s)$. By the same computations above, we obtain $M_\phi T_\delta=\tilde{\varphi}(D_{a}k_\delta D_{a})T_\delta$. Then we apply $M_\phi(T_\delta-T_0)$ to the constant function 1 on $K$ and conclude by taking the norm on $L_p(K)$.
\end{rmk}

\subsection{Proof of the theorem}

This section follows the ideas of \cite[\S 2]{Laf}, which were later used in \cite{LafdlS} and \cite{HaadL}. The letters $r,s,t$ will be now reserved for real numbers. Consider the set
\begin{equation*}
\Lambda=\{(r,s,t)\in\R^3\ |\ r\geq s \geq t,\ r+s+t=0\}.
\end{equation*}
Recall that the $K\overline{A^+}K$ decomposition of $\text{SL}(3,\R)$ states that, for every $g\in G$, there exists a unique $(r,s,t)\in\Lambda$ such that $g\in KD(r,s,t)K$, where
\begin{equation*}
D(r,s,t)=\left(\begin{matrix} e^r & 0 & 0\\ 0 & e^s & 0\\ 0 & 0 & e^t \end{matrix}\right).
\end{equation*}
Therefore, we may define functions $\gamma_1,\gamma_2,\gamma_3:G\to\R$ by
\begin{equation*}
g\in KD(\gamma_1(g),\gamma_2(g),\gamma_3(g))K,\quad (\gamma_1(g),\gamma_2(g),\gamma_3(g))\in\Lambda,\qquad \forall g\in G.
\end{equation*}
Observe that the eigenvalues of $g^tg$ are $e^{2\gamma_i(g)}$, $i=1,2,3$. So these functions are continuous. Moreover, for every $\varphi\in M_{p\text{-}cb}(G)$,
\begin{equation*}
\tilde{\varphi}(g)=\tilde{\varphi}(D(\gamma_1(g),\gamma_2(g),\gamma_3(g))),
\end{equation*}
where $\tilde{\varphi}$ is defined as in (\ref{avphi}).

\begin{lem}\label{Drst-Dt0t}
	For all $(r,s,t)\in\Lambda$ and $\varphi\in M_{p\text{-}cb}(G)$,
	\begin{equation}\label{epser2t}
	\left|\tilde{\varphi}(D(r,s,t)) - \tilde{\varphi}(D_{-t}k_0 D_{-t})\right| \leq \varepsilon(e^{r+2t})\|\varphi\|_{M_{p\text{-}cb}(G)}.
	\end{equation}
\end{lem}
\begin{proof}
	Observe first that, since the elements of $K$ are isometries of $\ell_2^3=(\R^3,\|\cdot\|_2)$, for all $k,k'\in K$,
	\begin{equation*}
	\|kD(r,s,t)^{-1}k'\|_{\mathcal{B}(\ell_2^3)}=\|D(r,s,t)^{-1}\|_{\mathcal{B}(\ell_2^3)}=e^{-t}.
	\end{equation*}
	Then
	\begin{align*}
	e^{-\gamma_3(D_ak_\delta D_a)} &= \|(D_ak_\delta D_a)^{-1}\|_{\mathcal{B}(\ell_2^3)}\\
	&= \left\|\left(\begin{matrix} \delta e^{2a} & -e^{\frac{a}{2}}\sqrt{1-\delta^2} & 0\\ e^{\frac{a}{2}}\sqrt{1-\delta^2} & \delta e^{-a} & 0\\ 0 & 0 & e^{-a} \end{matrix}\right)^{-1}\right\|_{\mathcal{B}(\ell_2^3)}\\
	&= \max\left\{e^a, \left\|\left(\begin{matrix} \delta e^{2a} & -e^{\frac{a}{2}}\sqrt{1-\delta^2}\\ e^{\frac{a}{2}}\sqrt{1-\delta^2} & \delta e^{-a} \end{matrix}\right)^{-1}\right\|_{\mathcal{B}(\ell_2^2)}\right\}.
	\end{align*}
	But
	\begin{equation*}
	\left(\begin{matrix} \delta e^{2a} & -e^{\frac{a}{2}}\sqrt{1-\delta^2}\\ e^{\frac{a}{2}}\sqrt{1-\delta^2} & \delta e^{-a} \end{matrix}\right)
	= \left(\begin{matrix} e^a & 0\\ 0 & e^{-\frac{a}{2}} \end{matrix}\right) k \left(\begin{matrix} e^a & 0\\ 0 & e^{-\frac{a}{2}} \end{matrix}\right)
	\end{equation*}
	with $k\in\text{SO}(2,\R)$. So, for $a>0$,
	\begin{equation*}
	\left\|\left(\begin{matrix} \delta e^{2a} & -e^{\frac{a}{2}}\sqrt{1-\delta^2}\\ e^{\frac{a}{2}}\sqrt{1-\delta^2} & \delta e^{-a} \end{matrix}\right)^{-1}\right\|_{\mathcal{B}(\ell_2^2)}\leq e^{\frac{a}{2}}e^{\frac{a}{2}}=e^a.
	\end{equation*}
	Therefore, if $a>0$, then $\gamma_3(D_ak_\delta D_a)=-a$ for all $\delta\in[-1,1]$. On the other hand,
	\begin{equation*}
	D_ak_1D_a=\left(\begin{matrix} 1 & 0 & 0\\ 0 & 1 & 0\\ 0 & 0 & 1 \end{matrix}\right)
	\left(\begin{matrix} e^{2a} & 0 & 0\\ 0 & e^{-a} & 0\\ 0 & 0 & e^{-a} \end{matrix}\right)
	\left(\begin{matrix} 1 & 0 & 0\\ 0 & 1 & 0\\ 0 & 0 & 1 \end{matrix}\right)
	\end{equation*}
	and
	\begin{equation*}
	D_ak_0D_a=\left(\begin{matrix} 0 & -1 & 0\\ 1 & 0 & 0\\ 0 & 0 & 1 \end{matrix}\right)
	\left(\begin{matrix} e^{\frac{a}{2}} & 0 & 0\\ 0 & e^{\frac{a}{2}} & 0\\ 0 & 0 & e^{-a} \end{matrix}\right)
	\left(\begin{matrix} 1 & 0 & 0\\ 0 & 1 & 0\\ 0 & 0 & 1 \end{matrix}\right).
	\end{equation*}
	Let $(r,s,t)\in\Lambda$. The previous equalities show that $\gamma_1(D_{-t}k_1 D_{-t})=-2t$ and $\gamma_1(D_{-t}k_0 D_{-t})=-\frac{t}{2}$. Moreover, observe that $-\frac{t}{2}\leq r\leq -2t$, so by continuity there exists $\delta\in[0,1]$ such that $\gamma_1(D_{-t}k_\delta D_{-t})=r$. Furthermore, since $\gamma_3(D_{-t}k_\delta D_{-t})=t$, we have $\gamma_2(D_{-t}k_\delta D_{-t})=s$. Thus, using Lemma \ref{phiDkD} and the equality $\tilde{\varphi}(D_{-t}k_\delta D_{-t})=\tilde{\varphi}(D(r,s,t))$, we get
	\begin{equation*}
	\left|\tilde{\varphi}(D(r,s,t)) - \tilde{\varphi}(D_{-t}k_0 D_{-t})\right|\leq \varepsilon(\delta)\|\varphi\|_{M_{p\text{-}cb}(G)}.
	\end{equation*}
	Now observe that
	\begin{equation*}
	\|D_{-t}k_\delta D_{-t}\|_{\mathcal{B}(\ell_2^3)}=\|D(r,s,t)\|_{\mathcal{B}(\ell_2^3)}=e^r
	\end{equation*}
	and
	\begin{equation*}
	\|D_{-t}k_\delta D_{-t}e_1\|^2=(\delta e^{-2t})^2+e^{-t}(1-\delta^2)\geq (\delta e^{-2t})^2.
	\end{equation*}
	This implies $\delta\leq e^{r+2t}$ and so we get (\ref{epser2t}).
\end{proof}

Let us now define a family of measures in $G$ by
\begin{equation}\label{defmg}
\int_G f\,dm_g = \int_K\int_K f(kgk')\,dk\,dk',\quad \forall f\in C_b(G),\ \forall g\in G.
\end{equation}
Since $M_{p\text{-}cb}(G)\hookrightarrow C_b(G)$ is contractive, $m_g$ defines an element of $M_{p\text{-}cb}(G)^*$ of norm at most $1$ for all $g\in G$. Then Lemma \ref{Drst-Dt0t} implies the following.

\begin{cor}\label{mg-mg'}
	Let $g,g'\in G$ such that $\gamma_3(g)=\gamma_3(g')$. Then
	\begin{equation*}
	\left\|m_{g}-m_{g'}\right\|_{M_{p\text{-}cb}(G)^*}\leq \varepsilon(e^{\gamma_1(g)+2\gamma_3(g)}) + \varepsilon(e^{\gamma_1(g')+2\gamma_3(g)}).
	\end{equation*}
\end{cor}
\begin{proof}
	Observe that, for every $\varphi\in M_{p\text{-}cb}(G)$,
	\begin{equation*}
	\langle m_g,\varphi\rangle=\int\varphi\,dm_g=\tilde{\varphi}(D(\gamma_1(g),\gamma_2(g),\gamma_3(g))).
	\end{equation*}
	Then Lemma \ref{Drst-Dt0t} together with the triangle inequality implies
	\begin{equation*}
	\left|\langle m_g-m_{g'},\varphi\rangle\right|\leq \left(\varepsilon(e^{\gamma_1(g)+2\gamma_3(g)}) + \varepsilon(e^{\gamma_1(g')+2\gamma_3(g)})\right)\|\varphi\|_{M_{p\text{-}cb}(G)},
	\end{equation*}
	and the result follows.
\end{proof}

Moreover, if instead of fixing $\gamma_3$, we fix $\gamma_1$, we obtain a similar result.

\begin{cor}\label{mg-mg'2}
	Let $g,g'\in G$ such that $\gamma_1(g)=\gamma_1(g')$. Then
	\begin{equation*}
	\left\|m_{g}-m_{g'}\right\|_{M_{p\text{-}cb}(G)^*}\leq \varepsilon(e^{-2\gamma_1(g)-\gamma_3(g)}) + \varepsilon(e^{-2\gamma_1(g)-\gamma_3(g)}).
	\end{equation*}
\end{cor}
\begin{proof}
	Consider the continuous group isomorphism $\theta:G\to G$ given by $\theta(g)=(g^t)^{-1}$. Then, for any $\varphi\in M_{p\text{-}cb}(G)$, we have that $\varphi\circ\theta\in M_{p\text{-}cb}(G)$ and $\|\varphi\circ\theta\|_{M_{p\text{-}cb}(G)}=\|\varphi\|_{M_{p\text{-}cb}(G)}$. Indeed,
	\begin{equation*}
	\varphi\circ\theta(g_1g_2^{-1})=\varphi(\theta(g_1)\theta(g_2)^{-1}) = \langle\beta(\theta(g_1)),\alpha(\theta(g_2))\rangle,\qquad \forall g_1,g_2\in G,
	\end{equation*}
	where $\alpha$ and $\beta$ are as in (\ref{phiab}). Let $(r,s,t), (r',s',t')\in\Lambda$ such that $r=r'$. We may use Corollary \ref{mg-mg'} to get
	\begin{multline*}
		\left|\widetilde{\varphi\circ\theta}(D(-t,-s,-r)) - \widetilde{\varphi\circ\theta}(D(-t',-s',-r'))\right|\\
	     \leq (\varepsilon(e^{-t-2r}) + \varepsilon(e^{-t'-2r}))\|\varphi\|_{M_{p\text{-}cb}(G)},
	\end{multline*}
	for every $\varphi\in M_{p\text{-}cb}(G)$. On the other hand, since $\theta$ restricted to $K$ is the identity,
	\begin{equation*}
	\widetilde{\varphi\circ\theta}(g)=\int_K\int_K\varphi(k\theta(g)k')\,dk\,dk'=\tilde{\varphi}\circ\theta(g),\quad \forall g\in G.
	\end{equation*}
	And
	\begin{align*}
		\theta(D(-t,-s,-r)) &=D(t,s,r)\\
		&=\left(\begin{matrix} 0 & 0 & -1\\ 0 & -1 & 0\\ -1 & 0 & 0 \end{matrix}\right)D(r,s,t)\left(\begin{matrix} 0 & 0 & -1\\ 0 & -1 & 0\\ -1 & 0 & 0 \end{matrix}\right).
	\end{align*}
	So
	\begin{equation*}
	\widetilde{\varphi\circ\theta}(D(-t,-s,-r))=\tilde{\varphi}(D(r,s,t)).
	\end{equation*}
	Therefore
	\begin{equation*}
	\left|\tilde{\varphi}(D(r,s,t)) - \tilde{\varphi}(D(r',s',t'))\right| \leq (\varepsilon(e^{-t-2r}) + \varepsilon(e^{-t'-2r}))\|\varphi\|_{M_{p\text{-}cb}(G)},
	\end{equation*}
	and the result follows as in Corollary \ref{mg-mg'}.
\end{proof}

We will use Corollaries \ref{mg-mg'} and \ref{mg-mg'2} repeatedly on some particular paths joining two points of $\Lambda$ in order to obtain the desired Cauchy net.

\begin{lem}\label{mgcauchy}
	There exists a constant $C>0$ such that, for all $g,g'\in G$ with $\gamma_3(g')\leq\gamma_3(g)$,
	\begin{equation*}
	\left\|m_{g}-m_{g'}\right\|_{M_{p\text{-}cb}(G)^*}\leq C e^{\frac{\gamma_3(g)}{p}}.
	\end{equation*}
\end{lem}
\begin{proof}
	Put $(r,s,t)=(\gamma_1(g),\gamma_2(g),\gamma_3(g))$ and $(r',s',t')=(\gamma_1(g'),\gamma_2(g'),\gamma_3(g'))$. We shall consider first the case $t\leq -1$. Let $n\geq 0$ such that
	\begin{equation*}
	t-2(n+1)< t' \leq t-2n.
	\end{equation*}
	Let $\varphi\in M_{p\text{-}cb}(G)$ with $\|\varphi\|_{M_{p\text{-}cb}(G)}=1$ and define $\phi(r,s,t)=\tilde{\varphi}(D(r,s,t))$. Assume first that $s=s'=-1$. Then 
		\begin{align*}
		\phi(1&-t,-1,t) -\phi(1-t',-1,t')\\
	=&\sum_{i=0}^{n-1} \phi(1-t+2i,-1,t-2i) - \phi(1-t+2i,1,t-2(i+1))\\
	&+ \sum_{i=0}^{n-1} \phi(1-t+2i,1,t-2(i+1)) - \phi(1-t+2(i+1),-1,t-2(i+1))\\
	&+ \phi(1-t+2n,-1,t-2n) - \phi(1-t+2n,-1+t-2n-t',t')\\
	&+ \phi(1-t+2n,-1+t-2n-t',t') - \phi(1-t',-1,t').
	\end{align*}
	Observe that $\varepsilon(e^{-a})=\tilde{C}e^{-\frac{a}{p}}$ for all $a\geq 0$. Then, using Corollaries \ref{mg-mg'} and \ref{mg-mg'2}, we get
	\begin{align*}
	|\phi(1-t+2i,-1,t-2i) - \qquad\qquad\qquad & \\
	 \phi(1-t+2i,1,t-2(i+1))| &\leq  \tilde{C}(e^{\frac{t-2i-2}{p}}+e^{\frac{t-2i}{p}})\\
	&\leq  2\tilde{C}e^{\frac{t}{p}}e^{\frac{-2i}{p}},
	\end{align*}
	\begin{align*}
	|\phi(1-t+2i,1,t-2(i+1)) \qquad\qquad\qquad & \\
	- \phi(1-t+2(i+1),-1,t-2(i+1))| &\leq \tilde{C}(e^{\frac{t-2i-3}{p}}+e^{\frac{t-2i-1}{p}})\\
	&\leq  2\tilde{C}e^{\frac{t}{p}}e^{\frac{-2i}{p}},
	\end{align*}
	and
	\begin{align*}
	|\phi(1-t+2n,-1,t-2n) \qquad\qquad\qquad & \\
	- \phi(1-t+2n,-1+t-2n-t',t')| &\leq \tilde{C}(e^{\frac{t-2n-2}{p}}+e^{\frac{2t-2-4n-t'}{p}})\\
	&\leq  \tilde{C}(e^{\frac{t-2n-2}{p}}+e^{\frac{t-2n}{p}})\\
	&\leq  2\tilde{C}e^{\frac{t}{p}}.
	\end{align*}
	Here we used the fact that $t-2(n+1)-t'<0$. We also get
	\begin{align*}
	|\phi(1-t+2n,-1+t-2n-t',t') \quad & \\
	- \phi(1-t',-1,t')| &\leq \tilde{C}(e^{\frac{1-t+2n+t'}{p}}+e^{\frac{1+t'}{p}})\\
	&\leq  2\tilde{C}e^{\frac{1+t'}{p}}\\
	&\leq  2\tilde{C}e^{\frac{1}{p}}e^{\frac{t}{p}}.
	\end{align*}
	Putting everything together, we obtain
	\begin{equation*}
	|\phi(1-t,-1,t)-\phi(1-t',-1,t')|\leq C_1e^{\frac{t}{p}},
	\end{equation*}
	with
	\begin{equation*}
	C_1=2\tilde{C}\left(\frac{2}{1-e^{\frac{-2}{p}}}+1+e^{\frac{1}{p}}\right),
	\end{equation*}
	and this constant depends only on $p$. Now we deal with the general case. If $s>-1$, then
	\begin{equation*}
	|\phi(r,s,t)-\phi(1-t,-1,t)|\leq \tilde{C} (e^{\frac{r+2t}{p}}+e^{\frac{1+t}{p}}) \leq 2\tilde{C}e^{\frac{1}{p}}e^{\frac{t}{p}}.
	\end{equation*}
	If $s<-1$,
	\begin{equation*}
	|\phi(r,s,t)-\phi(r,-1,1-r)|\leq \tilde{C} (e^{\frac{-2r-t}{p}}+e^{\frac{-r-1}{p}}) \leq 2\tilde{C}e^{\frac{-2}{p}}e^{\frac{t}{p}}.
	\end{equation*}
	In both cases we found $(\tilde{r},-1,\tilde{t})\in\Lambda$ satisfying
	\begin{equation*}
	|\phi(r,s,t)-\phi(\tilde{r},-1,\tilde{t})|\leq C_2 e^{\frac{t}{p}},
	\end{equation*}
	with $C_2$ depending only on $p$, and such that $\tilde{t}\leq t$. Therefore
	\begin{align*}
	|\phi(r,s,t) - \phi(r',s',t')| &\leq |\phi(r,s,t) - \phi(\tilde{r},-1,\tilde{t})| + |\phi(\tilde{r},-1,\tilde{t}) - \phi(\tilde{r}',-1,\tilde{t}')|\\
	& \quad\ + |\phi(\tilde{r}',-1,\tilde{t}') - \phi(r',s',t')|\\
	& \leq  C_2 e^{\frac{t}{p}} + C_1e^{\frac{1}{p}\max\{\tilde{t},\tilde{t}'\}} + C_2 e^{\frac{t'}{p}}\\
	& \leq  C_3 e^{\frac{t}{p}},
	\end{align*}
	with $C_3=(C_1+2C_2)$. This is valid when $t\leq -1$. If $t>-1$, then
	\begin{align*}
	|\phi(r,s,t) - \phi(r',s',t')| &\leq |\phi(r,s,t)| + |\phi(r',s',t')|\\
	& \leq 2 \|\varphi\|_{M_{p\text{-}cb}(G)} = 2\\
	& \leq C_4 e^{\frac{t}{p}},
	\end{align*}
	with $C_4=2e^{\frac{1}{p}}$. We obtain the result taking $C=\max\{C_3,C_4\}$.
\end{proof}

\begin{proof}[Proof of Theorem \ref{SL3pAP}]
	Let $p\in[2,\infty)$. Observe that, for $g\in G$, $\|g\|=e^{\gamma_1(g)}$ and $\gamma_3(g)\leq-\frac{1}{2}\gamma_1(g)$. So $\gamma_3(g)\to -\infty$ when $\|g\|\to\infty$. Then Lemma \ref{mgcauchy} implies the existence of $\mu\in M_{p\text{-}cb}(G)^*$ such that $m_g\to\mu$ as $\|g\|\to\infty$. Now define a new family of measures $\tilde{m}_g$ on $G$ by
	\begin{align}
		\int f\,d\tilde{m}_g &=\fint_{B_2}\int f(x)\,dm_{hg}(x)\, dh \notag\\
		&=\fint_{B_2}\int_K\int_K f(khgk')\,dk\,dk'\, dh,\quad\forall f\in C_c(G),\label{mtildg}
	\end{align}
	where $B_2=\{g\in G\ :\ \|g\|<2\}$ and $\fint_{B_2}\cdots dh$ stands for the normalized integration over $B_2$. These again are probability measures for each $g\in G$. Moreover, observe that equation (\ref{mtildg}) says that $\tilde{m}_g=\nu_K\ast \chi_{B_2}\ast \delta_g\ast\nu_K$, where $\nu_K$ is the normalized Haar measure on $K$ and $\chi_{B_2}$ is the normalized indicator function on $B_2$. Since $L_1(G)$ is a two-sided ideal in the convolution algebra of complex measures (see e.g. \cite[\S 2.5]{Foll}), we have that $\tilde{m}_g\in L_1(G)$ for all $g\in G$.
	On the other hand, using Lemma \ref{mgcauchy}, for all $\varphi\in M_{p\text{-}cb}(G)$ with $\|\varphi\|_{M_{p\text{-}cb}(G)}=1$,
	\begin{equation*}
	|\langle\mu-\tilde{m}_g,\varphi\rangle | \leq \fint_{B_2} |\langle\mu-m_{hg},\varphi\rangle |\,dh \leq \fint_{B_2} Ce^\frac{\gamma_3(hg)}{p}\,dh.
	\end{equation*}
	But $e^{\gamma_3(hg)}=\|(hg)^{-1}\|^{-1}$ and $\|g^{-1}\|\leq\|g^{-1}h^{-1}\|\|h\|$. Thus
	\begin{equation*}
	e^{\gamma_3(hg)}\leq \|h\|\|g^{-1}\|^{-1}\leq 2 e^{\gamma_3(g)}.
	\end{equation*}
	And so, $\tilde{m}_g$ converges to $\mu$ when $\|g\|\to\infty$. This implies that $\mu$ is in the adherence of $L_1(G)$ in $M_{p\text{-}cb}(G)^*$, that is, in $Q_{p\text{-}cb}(G)$. Now
	\begin{equation*}
	\langle\mu,1\rangle=\lim_{g\to\infty}m_g(G)=1.
	\end{equation*}
	And if $\varphi\in M_{p\text{-}cb}(G)\cap C_c(G)$, then for every $g$ sufficiently large, $\varphi(g)=0$. Since $\|kgk'\|=\|g\|$ for all $k,k'\in K$, we have also $\varphi(kgk')=0$. Therefore
	\begin{equation*}
	\langle\mu,\varphi\rangle=\lim_{g\to\infty}\int_K\int_K\varphi(kgk')\,dk\,dk'=0.
	\end{equation*}
	Thus, if $G$ had the $p\,$-AP, there would exist a net $(\varphi_i)$ in $A_{p,c}(G)$ such that
	\begin{equation*}
	0=\langle\mu,\varphi_i\rangle\to\langle\mu,1\rangle=1.
	\end{equation*}
	And this is a contradiction. Therefore, $G$ does not have the $p\,$-AP for any $p\in[2,\infty)$. By Proposition \ref{pAPiffp'AP}, the same holds for $1<p<2$.
\end{proof}

\section{The symplectic group $\text{Sp}(2,\R)$}

In this section we prove Theorem \ref{Sp2pAP}. The strategy is the same as for $\text{SL}(3,\R)$, with the main difference being the choice of the paths in the Weyl chamber in order to establish a result analogous to Lemma \ref{mgcauchy}. This is achieved by considering two different families of operators that will play the role of $(T_\delta)$, namely (\ref{defTtheta}) and (\ref{defStheta}). These operators were defined in \cite{Laa}. From now on we fix $p\in[2,\infty)$, $G=\text{Sp}(2,\R)$ and $K$ as in (\ref{defK=U2}).
Recall that $K$ is isomorphic to $\text{U}(2)$, and
\begin{equation*}
\text{SO}(2,\R)=\left\{\left(\begin{matrix}\cos\theta & -\sin\theta\\ \sin\theta & \cos\theta \end{matrix}\right)\ \Big|\ \theta\in [0,2\pi)\right\}
\end{equation*}
may be viewed as a subgroup of $\text{U}(2)$ by inclusion. Therefore, using the isomorphism (\ref{isomKU2}), we may define a family of operators $(T_\theta)$ on $L_p(K)$ by
\begin{equation}\label{defTtheta}
T_\theta f=\int\limits_{\text{SO}(2)}\int\limits_{\text{SO}(2)}\lambda(rd_\theta r')f\,dr\,dr',\quad \forall\theta\in [0,2\pi),\ \forall f\in L_p(K),
\end{equation}
where
\begin{equation*}
d_\theta=\left(\begin{matrix} e^{\ii\theta} & 0\\ 0 & e^{-\ii\theta}\end{matrix}\right)\in \text{SU}(2).
\end{equation*}
Here $\text{SO}(2)$ is endowed with the normalized Haar measure.

\begin{lem}
	There is a constant $C_1>0$ such that, for all $\theta\in\left[\frac{\pi}{6},\frac{\pi}{3}\right]$,
	\begin{equation*}
	\|T_{\theta}-T_{\frac{\pi}{4}}\|_{\mathcal{B}(L_p)} \leq  C_1\left|\theta-\tfrac{\pi}{4}\right|^{\frac{1}{p}}.
	\end{equation*}
\end{lem}
\begin{proof}
	The inequality for $p=2$ is given by \cite[Lemma 3.1]{Laa}. The other cases follow by interpolation as in Lemma \ref{estimTd}.
\end{proof}

We consider also
\begin{equation}\label{defStheta}
S_\theta=\frac{1}{2\pi}\int_0^{2\pi}\lambda(d_\varphi u_\theta d_{-\varphi})\,d\varphi, \quad\forall\theta\in\R,\ \forall f\in L_p(K),
\end{equation}
where
\begin{equation*}
u_\theta=\frac{1}{\sqrt{2}}\left(\begin{matrix} e^{\ii\theta} & -1\\ 1 & e^{-\ii\theta}\end{matrix}\right).
\end{equation*}

\begin{lem}
	There is a constant $C_2>0$ such that, for all $\theta_1,\theta_2\in\R$,
	\begin{equation*}
	\|S_{\theta_1}-S_{\theta_2}\|_{\mathcal{B}(L_p)} \leq  C_2\left|\theta_1-\theta_2\right|^{\frac{1}{2p}}.
	\end{equation*}
\end{lem}
\begin{proof}
	The case $p=2$ is given by \cite[Lemma 3.2]{Laa} and the others follow by interpolation as in Lemma \ref{estimTd}.
\end{proof}

Again since $PM_p(K)=CV_p(K)$, the operators $T_\theta$ and $S_\theta$ belong to $PM_p(K)$. In order to obtain a result analogous to Lemma \ref{phiDkD}, we define the following elements of $G$:
\begin{equation*}
D_a=\left(\begin{matrix} e^a & 0 & 0 & 0\\ 0 & 1 & 0 & 0\\ 0 & 0 & e^{-a} & 0\\ 0 & 0 & 0 & 1 \end{matrix} \right)\quad
D_a'=\left(\begin{matrix} e^a & 0 & 0 & 0\\ 0 & e^a & 0 & 0\\ 0 & 0 & e^{-a} & 0\\ 0 & 0 & 0 & e^{-a} \end{matrix} \right),\quad a\in\R.
\end{equation*}
Consider also
\begin{equation*}
v=\frac{1}{\sqrt{2}}\left(\begin{matrix} 1+\ii & 0\\ 0 & 1+\ii\end{matrix}\right)\in\text{U}(2),
\end{equation*}
which again, by the isomorphism (\ref{isomKU2}), defines an element of $K$. Then, defining the average $\tilde{\varphi}$ of a multiplier $\varphi\in M_{p\text{-}cb}(G)$ as in (\ref{avphi}), we obtain the following two lemmas, which can be proved in the same way as Lemma \ref{phiDkD}, using the operators $T_\theta$ and $S_\theta$ instead of $T_\delta$.

\begin{lem}\label{DadthvDa}
	For every $\varphi\in M_{p\text{-}cb}(G)$, $a\in\R$ and $\theta\in\left[\frac{\pi}{6},\frac{\pi}{3}\right]$,
	\begin{equation*}
	|\tilde{\varphi}(D_{a}'d_\theta vD_{a}')-\tilde{\varphi}(D_{2a}')|\leq C_1\left|\theta-\tfrac{\pi}{4}\right|^{\frac{1}{p}}\|\varphi\|_{M_{p\text{-}cb}(G)}.
	\end{equation*}
\end{lem}

\begin{lem}\label{DauthDa}
	For every $\varphi\in M_{p\text{-}cb}(G)$ and $a, \theta_1,\theta_2\in\R$,
	\begin{equation*}
	|\tilde{\varphi}(D_{a} u_{\theta_1} D_{a})-\tilde{\varphi}(D_{a} u_{\theta_2} D_{a})|\leq C_2\left|\theta_1-\theta_2\right|^{\frac{1}{2p}}\|\varphi\|_{M_{p\text{-}cb}(G)}.
	\end{equation*}
\end{lem}

The remaining of this section consists mostly of adaptations of some of the results in \cite[\S3]{HaadL}. Recall that $\overline{A^+}=\{D(\beta,\gamma)\ :\ \beta\geq\gamma\geq 0\}$, where $D(\beta,\gamma)$ is defined as in (\ref{Dbetagamma}).

\begin{lem}{\cite[Lemma 3.16]{HaadL}}\label{systeq1}
	Let $\beta\geq\gamma\geq 0$. Then the equations
	\begin{align*}
		\sinh^2(2s)+\sinh^2 (s) &= \sinh^2(\beta)+\sinh^2 (\gamma)\\
		\sinh(2t)\sinh(t) &= \sinh(\beta)\sinh(\gamma)
	\end{align*}
	have unique solutions $s=s(\beta,\gamma)$, $t=t(\beta,\gamma)$ in the interval $[0,\infty)$. Moreover,
	\begin{equation*}
	s\geq\frac{\beta}{4},\quad t\geq\frac{\gamma}{2}.
	\end{equation*}
\end{lem}

\begin{lem}{\cite[Lemma 3.19]{HaadL}}\label{systeq2}
	Let $s\geq t\geq 0$. Then the system of equations
	\begin{align*}
		\sinh^2(\beta)+\sinh^2 (\gamma) &= \sinh^2(2s)+\sinh^2 (s)\\
		\sinh(\beta)\sinh(\gamma) &= \sinh(2t)\sinh(t)
	\end{align*}
	has a unique solution $(\beta,\gamma)\in\R^2$ for which $\beta\geq\gamma\geq 0$. Moreover, if $1\leq t\leq s\leq \frac{3t}{2}$, then
	\begin{align*}
		|\beta-2s|&\leq 1,\\
		|\gamma+2s-3t|&\leq 1.
	\end{align*}
\end{lem}

\begin{lem}\label{ineqeps1}
	There exists a constant $C_3>0$ such that whenever $\beta\geq\gamma\geq 0$ and $s=s(\beta,\gamma)$ is chosen as in Lemma \ref{systeq1}, then
	\begin{equation*}
	|\tilde{\varphi}(D(\beta,\gamma))-\tilde{\varphi}(D(2s,s))|\leq C_3e^{-\frac{\beta-\gamma}{4p}}\|\varphi\|_{M_{p\text{-}cb}(G)},
	\end{equation*}
	for every $\varphi\in M_{p\text{-}cb}(G)$.
\end{lem}
\begin{proof}
	First assume that $\beta-\gamma\geq 8$. Then by (the proof of) \cite[Lemma 3.17]{HaadL}, there exist $r_1,r_2\in\left[0,\frac{1}{2}\right]$ and $w_1,w_2\in\text{SU}(2)$ given by
	\begin{equation*}
	w_i=\left(\begin{matrix} a_i+\ii b_i & 0\\ 0 & a_i-\ii b_i\end{matrix}\right),\quad i=1,2,
	\end{equation*}
	where $a_i=\left(\dfrac{1+r_i}{2}\right)^{\frac{1}{2}}$ and $b_i=\left(\dfrac{1-r_i}{2}\right)^{\frac{1}{2}}$, such that
	\begin{equation*}
	D_a'w_1vD_a'\in KD(\beta,\gamma)K,\quad D_a'w_2vD_a'\in KD(2s,s)K,
	\end{equation*}
	for some $a>0$. Observe that $w_i=d_{\theta_i}$ with $\theta_i=\arctan\left(\frac{b_i}{a_i}\right)$, and 
	\begin{equation*}
	(\tan\theta_i)^2=\frac{1-r_i}{1+r_i}\in\left[\tfrac{1}{3},1\right]
	\end{equation*}
	because $0\leq r_i\leq\frac{1}{2}$. This implies that $\theta_i\in\left[\frac{\pi}{6},\frac{\pi}{4}\right]$, so by Lemma \ref{DadthvDa}, 
	\begin{equation*}
	|\tilde{\varphi}(D_{a}'w_i vD_{a}')-\tilde{\varphi}(D_{2a}')|\leq C_1\left|\theta_i-\tfrac{\pi}{4}\right|^{\frac{1}{p}}\|\varphi\|_{M_{p\text{-}cb}(G)},
	\end{equation*}
	for every $\varphi\in M_{p\text{-}cb}(G)$. Using the inequality $|\arctan(x)-\arctan(y)|\leq |x-y|$, we obtain
	\begin{equation*}
	\left|\theta_i-\tfrac{\pi}{4}\right|\leq \left|\tfrac{b_i}{a_i}-1\right|=1-\left(\frac{1-r_i}{1+r_i}\right)^{\frac{1}{2}}\leq 2r_i.
	\end{equation*}
	The last inequality can be justified as follows:
	\begin{align*}
	1-2r_i\leq \left(\frac{1-r_i}{1+r_i}\right)^{\frac{1}{2}} &\iff  1-4r_i+4r_i^2\leq\frac{1-r_i}{1+r_i}\\
	&\iff  1-4r_i+4r_i^2+r_i-4r_i^2+4r_i^3\leq 1-r_i\\
	&\iff  4r_i^3\leq 2r_i \iff r_i^2\leq\frac{1}{2}.
	\end{align*}
	Therefore, since $\tilde{\varphi}(D(\beta,\gamma))=\tilde{\varphi}(D_{a}'w_1 vD_{a}')$ and $\tilde{\varphi}(D(2s,s))=\tilde{\varphi}(D_{a}'w_2 vD_{a}')$,
	\begin{align*}
	|\tilde{\varphi}(D(\beta,\gamma))-\tilde{\varphi}(D(2s,s))|
	&\leq  |\tilde{\varphi}(D_{a}'w_1 vD_{a}')-\tilde{\varphi}(D_{2a}')|\\
	&\qquad + |\tilde{\varphi}(D_{a}'w_2 vD_{a}')-\tilde{\varphi}(D_{2a}')|\\
	&\leq  C_1\left(\left(2r_1\right)^{\frac{1}{p}}+\left(2r_2\right)^{\frac{1}{p}}\right)\|\varphi\|_{M_{p\text{-}cb}(G)}.
	\end{align*}
	Again by the proof of \cite[Lemma 3.17]{HaadL}, $r_1,r_2\leq 2e^{\frac{\gamma-\beta}{4}}$. So
	\begin{equation*}
	C_1\left(\left(2r_1\right)^{\frac{1}{p}}+\left(2r_2\right)^{\frac{1}{p}}\right)\leq C e^{\frac{\gamma-\beta}{4p}},
	\end{equation*}
	where $C>0$ depends only on $p$. Finally, if $\beta-\gamma\leq 8$, then $\frac{2}{p}+\frac{\gamma-\beta}{4p}\geq 0$ and 
	\begin{align*}
		|\tilde{\varphi}(D(\beta,\gamma))-\tilde{\varphi}(D(2s,s))| &\leq 2 \|\tilde{\varphi}\|_{\infty}\\
		&\leq 2 \|\tilde{\varphi}\|_{M_{p\text{-}cb}(K)}\\
		&\leq 2 \|\varphi\|_{M_{p\text{-}cb}(G)}\\
		&\leq C' e^{\frac{\gamma-\beta}{4p}}\|\varphi\|_{M_{p\text{-}cb}(G)},
	\end{align*}
	with $C'=2e^{\frac{2}{p}}$. Hence, the result follows with $C_3=\max\{C,C'\}$.
\end{proof}

\begin{lem}\label{ineqeps2}
	There exists a constant $C_4>0$ such that whenever $\beta\geq\gamma\geq 0$ and $t=t(\beta,\gamma)$ is chosen as in Lemma \ref{systeq1}, then
	\begin{equation*}
	|\tilde{\varphi}(D(\beta,\gamma))-\tilde{\varphi}(D(2t,t))|\leq C_4e^{-\frac{\gamma}{4p}}\|\varphi\|_{M_{p\text{-}cb}(G)},
	\end{equation*}
	for every $\varphi\in M_{p\text{-}cb}(G)$.
\end{lem}
\begin{proof}
	Assume first that $\gamma\geq 2$. Then the proof of \cite[Lemma 3.18]{HaadL} gives the existence of $\theta_1,\theta_2\in\R$ such that $|\theta_1-\theta_2|\leq e^{-\frac{\gamma}{2}}$ and
	\begin{equation*}
	D_au_{\theta_1}D_a\in KD(\beta,\gamma)K,\quad D_au_{\theta_2}D_a\in KD(2t,t)K,
	\end{equation*}
	for some $a>0$. Then, using Lemma \ref{DauthDa} we get 
	\begin{align*}
	|\tilde{\varphi}(D(\beta,\gamma))-\tilde{\varphi}(D(2t,t))| &= |\tilde{\varphi}(D_au_{\theta_1}D_a)-\tilde{\varphi}(D_au_{\theta_2}D_a)|\\
	&\leq  C_2\left|\theta_1-\theta_2\right|^{\frac{1}{2p}}\|\varphi\|_{M_{p\text{-}cb}(G)}\\
	&\leq  C e^{-\frac{\gamma}{4p}}\|\varphi\|_{M_{p\text{-}cb}(G)}.
	\end{align*}
	for every $\varphi\in M_{p\text{-}cb}(G)$, with $C>0$ depending only on $p$. Finally, if $\gamma< 2$, then
	\begin{equation*}
	|\tilde{\varphi}(D(\beta,\gamma))-\tilde{\varphi}(D(2t,t))|  \leq 2 \|\tilde{\varphi}\|_{\infty}\leq 2 \|\varphi\|_{M_{p\text{-}cb}(G)}\leq C' e^{-\frac{\gamma}{4p}}\|\varphi\|_{M_{p\text{-}cb}(G)},
	\end{equation*}
	with $C'=2e^{\frac{1}{2p}}$. The result follows with $C_4=\max\{C,C'\}$.
\end{proof}

\begin{lem}\label{propC5}
	There exists a constant $C_5>0$ such that whenever $s,t\geq 0$ satisfy $2\leq t\leq s\leq\frac{6}{5}t$,
	\begin{equation*}
	|\tilde{\varphi}(D(2s,s))-\tilde{\varphi}(D(2t,t))|\leq C_5e^{-\frac{s}{8p}}\|\varphi\|_{M_{p\text{-}cb}(G)},
	\end{equation*}
	for every $\varphi\in M_{p\text{-}cb}(G)$.
\end{lem}
\begin{proof}
	Take $\beta\geq\gamma\geq 0$ given by Lemma \ref{systeq2}. Then, by Propositions \ref{ineqeps1} and \ref{ineqeps2},
	\begin{equation*}
	|\tilde{\varphi}(D(\beta,\gamma))-\tilde{\varphi}(D(2s,s))|\leq C_3e^{-\frac{\beta-\gamma}{4p}}\|\varphi\|_{M_{p\text{-}cb}(G)},
	\end{equation*}
	\begin{equation*}
	|\tilde{\varphi}(D(\beta,\gamma))-\tilde{\varphi}(D(2t,t))|\leq C_4e^{-\frac{\gamma}{4p}}\|\varphi\|_{M_{p\text{-}cb}(G)}.
	\end{equation*}
	As shown in \cite[Lemma 3.20]{HaadL}, $\min\{\gamma,\beta-\gamma\}\geq \frac{s}{2}-1$. Thus the result follows by the triangle inequality with $C_5=e^{\frac{1}{4p}}(C_3+C_4)$.
\end{proof}

\begin{lem}\label{propC6}
	There exists a constant $C_6>0$ such that whenever $s\geq t\geq 0$,
	\begin{equation*}
	|\tilde{\varphi}(D(2s,s))-\tilde{\varphi}(D(2t,t))|\leq C_6e^{-\frac{t}{8p}}\|\varphi\|_{M_{p\text{-}cb}(G)}, 
	\end{equation*}
	for every $\varphi\in M_{p\text{-}cb}(G)$.
\end{lem}
\begin{proof}
	Assume first that $t\geq 5$. Write $s=t+n+\delta$ where $n\geq 0$ is an integer and $\delta\in[0,1)$. Then, by Lemma \ref{propC5}, for every $j\in\{0,1,...,n-1\}$, 
	\begin{equation*}
	|\tilde{\varphi}(D(2(t+j+1),t+j+1))-\tilde{\varphi}(D(2(t+j),t+j))|\leq C_5e^{-\frac{t+j+1}{8p}}\|\varphi\|_{M_{p\text{-}cb}(G)}.
	\end{equation*}
	And
	\begin{equation*}
	|\tilde{\varphi}(D(2s,s))-\tilde{\varphi}(D(2(t+n),t+n))|\leq C_5e^{-\frac{s}{8p}}\|\varphi\|_{M_{p\text{-}cb}(G)}.
	\end{equation*}
	Thus
	\begin{align*}
	|\tilde{\varphi}(D(2s,s))-\tilde{\varphi}(D(2t,t))| &\leq  C_5\left(\sum_{j=0}^n e^{-\frac{t+j}{8p}}\right)\|\varphi\|_{M_{p\text{-}cb}(G)}\\
	&\leq  C e^{-\frac{t}{8p}}\|\varphi\|_{M_{p\text{-}cb}(G)},
	\end{align*}
	with $C>0$ depending only on $p$. Finally, if $t<5$, then
	\begin{equation*}
	|\tilde{\varphi}(D(2s,s))-\tilde{\varphi}(D(2t,t))|\leq 2\|\tilde{\varphi}\|_{\infty}\leq 2\|\varphi\|_{M_{p\text{-}cb}(G)}.
	\end{equation*}
	The result follows by taking $C_6=\max\{C,2e^{\frac{5}{8p}}\}$.
\end{proof}

\begin{lem}\label{propC7}
	There exists a constant $C_7>0$ such that for all $\beta_1\geq\gamma_1\geq 0$, $\beta_2\geq\gamma_2\geq 0$ with $\beta_1\leq\beta_2$,
	\begin{equation*}
	|\tilde{\varphi}(D(\beta_1,\gamma_1))-\tilde{\varphi}(D(\beta_2,\gamma_2))|\leq C_7e^{-\frac{\beta_1}{32p}}\|\varphi\|_{M_{p\text{-}cb}(G)}, 
	\end{equation*}
	for every $\varphi\in M_{p\text{-}cb}(G)$.
\end{lem}
\begin{proof}
	If $\beta_i\geq 2\gamma_i$, then $\beta_i-\gamma_i\geq \frac{\beta_i}{2}$, so by Lemma \ref{ineqeps1},
	\begin{equation*}
	|\tilde{\varphi}(D(\beta_i,\gamma_i))-\tilde{\varphi}(D(2s_i,s_i))|\leq C_3e^{-\frac{\beta_i}{8p}}\|\varphi\|_{M_{p\text{-}cb}(G)},
	\end{equation*}
	with $s_i\geq\frac{\beta_i}{4}$. If $\beta_i<2\gamma_i$, by Lemma \ref{ineqeps2},
	\begin{equation*}
	|\tilde{\varphi}(D(\beta_i,\gamma_i))-\tilde{\varphi}(D(2t_i,t_i))|\leq C_4e^{-\frac{\beta_i}{8p}}\|\varphi\|_{M_{p\text{-}cb}(G)},
	\end{equation*}
	with $t_i\geq\frac{\gamma_i}{2}\geq\frac{\beta_i}{4}$. In both cases there exists $r_i\geq\frac{\beta_i}{4}$ such that
	\begin{equation*}
	|\tilde{\varphi}(D(\beta_i,\gamma_i))-\tilde{\varphi}(D(2r_i,r_i))|\leq Ce^{-\frac{\beta_i}{8p}}\|\varphi\|_{M_{p\text{-}cb}(G)},
	\end{equation*}
	with $C=\max\{C_3,C_4\}$. Moreover, by Lemma \ref{propC6},
	\begin{equation*}
	|\tilde{\varphi}(D(2r_1,r_1))-\tilde{\varphi}(D(2r_2,r_2))|\leq C_6e^{-\frac{r}{8p}}\|\varphi\|_{M_{p\text{-}cb}(G)},
	\end{equation*}
	with $r=\min\{r_1,r_2\}$. Observe that $r\geq\min\{\frac{\beta_1}{4},\frac{\beta_2}{4}\}=\frac{\beta_1}{4}$. Putting everything together we get
	\begin{equation*}
	|\tilde{\varphi}(D(\beta_1,\gamma_1))-\tilde{\varphi}(D(\beta_2,\gamma_2))|\leq C_7e^{-\frac{\beta_1}{32p}}\|\varphi\|_{M_{p\text{-}cb}(G)},
	\end{equation*}
	with $C_7=2C+C_6$.
\end{proof}

Now we are in position of proving Theorem \ref{Sp2pAP} in the same way as Theorem \ref{SL3pAP}. Define the functions $\beta,\gamma:G\to[0,\infty)$ by 
\begin{equation*}
g\in KD(\beta(g),\gamma(g))K,\quad \beta(g)\geq\gamma(g)\geq 0,\quad\forall g\in G,
\end{equation*}
and the family of measures $m_g$ as in (\ref{defmg}). Then Lemma \ref{propC7} can be stated as follows.

\begin{cor}\label{mgcauchy2}
	There exists a constant $C>0$ such that, for all $g,g'\in G$ with $\beta(g)\leq\beta(g')$,
	\begin{equation*}
	\left\|m_{g}-m_{g'}\right\|_{M_{p\text{-}cb}(G)^*}\leq C e^{-\frac{\beta(g)}{32p}}.
	\end{equation*}
\end{cor}

\begin{proof}[Proof of Theorem \ref{Sp2pAP}]
	The proof is almost the same as that of Theorem \ref{SL3pAP}, by using Corollary \ref{mgcauchy2} instead of Lemma \ref{mgcauchy}. Just observe that in this case,  $\|g\|=e^{\beta(g)}$ and
	\begin{equation*}
	e^{-\beta(hg)}=\|hg\|^{-1}\leq\|h^{-1}\|\|g\|^{-1} \leq \|h\| e^{-\beta(g)},
	\end{equation*}
	for every $g,h\in G$. The last inequality holds because $h^{-1}=J^{-1}h^tJ$, where $J$ is the matrix defined in (\ref{defJ}).
\end{proof}

\section{Simple Lie groups with finite center}

Now we state the proof of Theorem \ref{LiepAP} using Theorems \ref{SL3pAP} and \ref{Sp2pAP}. For this purpose, we first need to discuss some stability properties of the $p\,$-AP.

\subsection{Stability properties}

Let $G$ be a second countable locally compact group and let $1<p<\infty$.
\begin{prop}\label{G/KpAP}
	If $K$ is a compact normal subgroup of $G$, then $G/K$ has the $p\,$-AP if and only if $G$ has the $p\,$-AP.
\end{prop}
\begin{proof}
	Suppose first that $G/K$ has the $p\,$-AP. Define a linear map $\Psi:M_{p\text{-}cb}(G/K)\to M_{p\text{-}cb}(G)$ by $\Psi(\varphi)(s)=\varphi(\dot{s})$. It is well defined because for every $\varphi\in M_{p\text{-}cb}(G/K)$, we can write
	\begin{equation*}
	\Psi(\varphi)(st^{-1})=\langle\beta\circ\pi(s),\alpha\circ\pi(t)\rangle,
	\end{equation*}
	where $\alpha$ and $\beta$ are as in (\ref{phiab}) and $\pi:G\to G/K$ is the quotient map, which is continuous. Hence, taking the infimum over all $\alpha,\beta$, we see that $\|\Psi(\varphi)\|_{M_{p\text{-}cb}(G)}\leq \|\varphi\|_{M_{p\text{-}cb}(G/K)}$. Furthermore, we shall prove that $\Psi$ is $\sigma(M_{p\text{-}cb}(G/K),Q_{p\text{-}cb}(G/K))-\sigma(M_{p\text{-}cb}(G),Q_{p\text{-}cb}(G))$ continuous. Consider the linear map $T:C_c(G)\to C_c(G/K)$ given by
	\begin{equation*}
	T(f)(\dot{s})=\int_K f(sk)\,dk,\quad \forall f\in C_c(G),
	\end{equation*}
	where $\dot{s}=\pi(s)$. This map is well defined and does not depend on the choice of the representative of $\dot{s}$. As shown in \cite[\S 3.4]{Rei}, we may fix the Haar measure on $G/K$ so that $T$ extends to a contraction from $L_1(G)$ onto $L_1(G/K)$, and
	\begin{equation}\label{intG/K=intG}
	\int_{G/K}T(f)(\dot{s})\,d\dot{s} = \int_G f(s)\,ds,
	\end{equation}
	for every $f\in L_1(G)$. We wish now to extend $T$ to $Q_{p\text{-}cb}(G)$. Take $\varphi\in M_{p\text{-}cb}(G/K)$ and $f\in L_1(G)$. Observe that $\Psi(\varphi)(sk)=\Psi(\varphi)(s)$ for all $s\in G$, $k\in K$. Then, regarding $T(f)$ as an element of $M_{p\text{-}cb}(G/K)^*$ as in (\ref{alphacb}), we get
	\begin{align*}
	\langle T(f),\varphi\rangle &= \int_{G/K}T(f)(\dot{s})\varphi(\dot{s})\,d\dot{s}\\
	&= \int_{G/K}\left(\int_K f(sk)\,dk\right)\Psi(\varphi)(s)\,d\dot{s}\\
	&= \int_{G/K}\int_K f(sk)\Psi(\varphi)(sk)\,dk\,d\dot{s} \\
	&= \int_{G/K}T\left(f \Psi(\varphi)\right)(\dot{s})\,d\dot{s}\\
	&= \int_G f(s) \Psi(\varphi)(s)\,ds \\
	&= \langle f,\Psi(\varphi)\rangle.
	\end{align*}
	Here we have used the fact that $\Psi(\varphi)\in C_b(G)$, which implies that $f \Psi(\varphi)\in L_1(G)$. The previous computation shows two things. First, it implies that $T$ may be extended to a map from $Q_{p\text{-}cb}(G)$ onto $Q_{p\text{-}cb}(G/K)$, since
	\begin{equation*}
	\left|\langle T(f),\varphi\rangle\right| \leq \|f\|_{M_{p\text{-}cb}(G)^*}\|\Psi(\varphi)\|_{M_{p\text{-}cb}(G)}\leq \|f\|_{M_{p\text{-}cb}(G)^*}\|\varphi\|_{M_{p\text{-}cb}(G/K)}.
	\end{equation*}
	Second, it shows that $\Psi=T^*$. Hence, $\Psi$ is weak*-weak* continuous. Now take a net $(\varphi_i)$ in $A_p(G/K)$ such that $\varphi_i\to 1$ in $\sigma(M_{p\text{-}cb}(G/K),Q_{p\text{-}cb}(G/K))$. We have that $\Psi(\varphi_i)\to\Psi(1)$ in $\sigma(M_{p\text{-}cb}(G),Q_{p\text{-}cb}(G))$. Moreover, since $K$ is compact, by \cite[Proposition 6]{Her2}, $\Psi(\varphi_i)\in A_p(G)$ for all $i$. Since $\Psi(1)$ is the constant function $1$ on $G$, we conclude that $G$ has the $p\,$-AP.\\
	The proof of the other direction follows the same idea. We shall define operators $\tilde{T}$ and $\tilde{\Psi}$ by the same formulas as $T$ and $\Psi$, but in different spaces. Consider $\tilde{T}: M_{p\text{-}cb}(G)\to M_{p\text{-}cb}(G/K)$ given by 
	\begin{equation*}
	\tilde{T}(\varphi)(\dot{s})=\int_K \varphi(sk)\,dk,\quad \forall \varphi\in M_{p\text{-}cb}(G).
	\end{equation*}
	Let us prove that this map is well defined and continuous. Take $\varphi\in M_{p\text{-}cb}(G)$ and functions $\alpha:G\to E$, $\beta:G\to E^*$ as in (\ref{phiab}). Thus
	\begin{equation*}
	\tilde{T}(\varphi)(\dot{s}\dot{t}^{-1})=\int_K \langle\beta(s),\alpha(k^{-1}t)\rangle\,dk= \left\langle\beta(s),\int_K\alpha(k^{-1}t)\,dk\right\rangle,\quad\forall s,t\in G.
	\end{equation*}
	Define $\tilde{\alpha}(\dot{t})=\int_K\alpha(k^{-1}t)\,dk$. To see that it is well defined, take $t\in G$, $k'\in K$ and recall that $tk't^{-1}\in K$ because $K$ is normal. Then
	\begin{equation*}
	\int_K\alpha(k^{-1}tk')\,dk=\int_K\alpha(k^{-1}tk't^{-t}t)\,dk=\int_K\alpha(k^{-1}t)\,dk.
	\end{equation*}
	Now let $F$ be the closed subspace of $E$ generated by $\{\tilde{\alpha}(\dot{t})\ :\ t\in G\}$. Observe that $F\in SQ_p$ and that $\tilde{\alpha}:G/K\to F$ is a continuous function because $t\mapsto \int_K\alpha(k^{-1}t)\,dk$ is continuous thanks to Lemma \ref{lemcontint}. Now define a function $\tilde{\beta}:G/K\to F^*$ by
	\begin{equation*}
	\langle\beta(\dot{s}),x\rangle=\langle\beta(s),x\rangle,\quad \forall x\in F.
	\end{equation*}
	Again, it is well defined because
	\begin{align*}
	\left\langle\beta(sk'),\int_K\alpha(k^{-1}t)\,dk\right\rangle &= \int_K\varphi(sk't^{-1}k)\,dk \\
	&= \int_K\varphi(st^{-1}tk't^{-1}k)\,dk\\
	& = \int_K\varphi(st^{-1}k)\,dk \\
	&= \left\langle\beta(s),\int_K\alpha(k^{-1}t)\,dk\right\rangle,
	\end{align*}
	for all $s,t\in G$ and $k'\in K$. Moreover, it is continuous since, by definition,
	\begin{equation*}
	\|\tilde{\beta}(\dot{s}_1)-\tilde{\beta}(\dot{s}_2)\|_{F^*}\leq\|\beta(s_1)-\beta(s_2)\|_{E^*},\quad\forall s_1,s_2\in G.
	\end{equation*}
	Hence
	\begin{equation*}
	\tilde{T}(\varphi)(\dot{s}\dot{t}^{-1})=\langle\beta(\dot{s}),\tilde{\alpha}(\dot{t})\rangle,\quad\forall \dot{s},\dot{t}\in G/K,
	\end{equation*}
	and $\|\tilde{T}(\varphi)\|_{M_{p\text{-}cb}(G/K)}\leq \|\tilde{\alpha}\|_\infty\|\tilde{\beta}\|_\infty\leq \|\alpha\|_\infty\|\beta\|_\infty$. Thus $\|\tilde{T}(\varphi)\|_{M_{p\text{-}cb}(G/K)}\leq\|\varphi\|_{M_{p\text{-}cb}(G)}$. Now define $\tilde{\Psi}:C_c(G/K)\to C_c(G)$ by $\tilde{\Psi}(f)(s)=f(\dot{s})$. We wish to extend this map to $Q_{p\text{-}cb}(G/K)$ and prove that $\tilde{\Psi}^*=\tilde{T}$. Observe that by (\ref{intG/K=intG}),
	\begin{align*}
		\int_G|\tilde{\Psi}(f)(s)|\,ds &=\int_{G/K}T\left(|\tilde{\Psi}(f)|\right)(\dot{s})\,d\dot{s}\\
		&= \int_{G/K}\int_K\left|f(\dot{s})\right|\,dk\,d\dot{s}\\
		&=\int_{G/K}\left|f(\dot{s})\right|\,d\dot{s},
	\end{align*}
	for every $f\in C_c(G/K)$. Thus, $\tilde{\Psi}$ extends to an isometry from $L_1(G/K)$ into $L_1(G)$. Furthermore, for all $\varphi\in M_{p\text{-}cb}(G)$ and $f\in C_c(G/K)$,
	\begin{align*}
	\langle\tilde{T}(\varphi),f\rangle &= \int_{G/K}\tilde{T}(\varphi)(\dot{s})\tilde{\Psi}(f)(s)\,d\dot{s} \\
	&= \int_{G/K}\int_K\varphi(sk)\tilde{\Psi}(f)(sk)\,dk\,d\dot{s}\\
	&= \int_G\varphi(s)\tilde{\Psi}(f)(s)\,ds \\
	&= \langle\varphi,\tilde{\Psi}(f)\rangle.
	\end{align*}
	Hence,
	\begin{equation*}
	|\langle\varphi,\tilde{\Psi}(f)\rangle|\leq \|\tilde{T}(\varphi)\|_{M_{p\text{-}cb}(G/K)}\|f\|_{M_{p\text{-}cb}(G/K)^*}\leq \|\varphi\|_{M_{p\text{-}cb}(G)}\|f\|_{M_{p\text{-}cb}(G/K)^*}.
	\end{equation*}
	Thus, $\tilde{\Psi}$ extends to a contraction from $Q_{p\text{-}cb}(G/K)$ to $Q_{p\text{-}cb}(G)$, and $\tilde{\Psi}^*=\tilde{T}$. Finally, by \cite[Proposition 6]{Her2}, $\tilde{T}$ maps $A_p(G)$ to $A_p(G/K)$, so we conclude as before.
\end{proof}

\begin{rmk}
	As one of the referees pointed out, the proof of Proposition \ref{G/KpAP} can be made more conceptual. Namely, if $\nu_K$ stands for the normalized Haar measure on $K$, viewed as a measure on $G$, then it satifies
	$\nu_K=\nu_K\ast\nu_K=\check{\nu}_K=\delta_s\ast\nu_K\ast\delta_{s^{-1}},$
	and there is a standard identification $\nu_K\ast C_c(G)=C_c(G/K)$ which yields an identification $\nu_K\ast L_p(G)=L_p(G/K)$ and thus $\nu_K\ast A_p(G)=A_p(G/K)$. By similar computations as those of the previous proof, we also find $\nu_K\ast M_{p\text{-}cb}(G)=M_{p\text{-}cb}(G/K)$ and $\nu_K\ast Q_{p\text{-}cb}(G)=Q_{p\text{-}cb}(G/K)$. Therefore, the equivalence of the $p\,$-AP for $G$ and $G/K$ follows from all these identifications.
\end{rmk}

\begin{prop}\label{HpAP}
	If $G$ has the $p\,$-AP, then every closed subgroup of $G$ has the $p\,$-AP.
\end{prop}
\begin{proof}
	Let $H$ be a closed subgroup of $G$ endowed with a left Haar measure. As in the proof of Proposition \ref{G/KpAP}, we shall define a map from $Q_{p\text{-}cb}(H)$ to $Q_{p\text{-}cb}(G)$ and then consider its adjoint. Take $\psi\in C_c(G)$ such that $\psi\geq 0$ and $\int_G\psi=1$, and consider the convolution map $\Phi:L_1(H)\to L_1(G)$ given by
	\begin{equation*}
	\Phi(f)(s)=\int_H f(t)\psi(t^{-1}s)\,dt.
	\end{equation*}
	This map is well defined since $\Phi(f)=\mu_f\ast\psi$, where $\mu_f$ is the measure on $G$ given by
	\begin{equation*}
	\int_G g(s)\,d\mu_f(s) = \int_H g(t)f(t) dt, \quad \forall g\in C_c(G).
	\end{equation*}
	Moreover, it is a contraction. Indeed, for every $f\in L_1(H)$,
	\begin{equation*}
	\left\|\Phi(f)\right\|_{L_1(G)}\leq \int_H |f(t)| \int_G |\psi(t^{-1}s)|\,ds\,dt\leq \|\psi\|_{L_1(G)}\|f\|_{L_1(H)}= \|f\|_{L_1(H)}.
	\end{equation*}
	Take now $\varphi\in M_{p\text{-}cb}(G)$.
	\begin{equation*}
	\langle\Phi(f),\varphi\rangle =\int_G\left(\int_H f(t)\psi(t^{-1}s)\,dt\right)\varphi(s)\,ds.
	\end{equation*}
	Observe that
	\begin{equation*}
	\int_G\int_H |f(t)| |\psi(t^{-1}s)| |\varphi(s)|\,dt\,ds \leq \|f\|_{L_1(H)}\|\varphi\|_\infty.
	\end{equation*}
	So by Fubini's theorem,
	\begin{equation}\label{Phi(f),phi}
	\langle\Phi(f),\varphi\rangle =\int_H f(t) \int_G\psi(t^{-1}s)\varphi(s)\,ds\,dt=\int_H f(t) \ \varphi\ast\check{\psi}(t)\,dt.
	\end{equation}
	Let us prove now that $\varphi\ast\check{\psi}$ defines an element of $M_{p\text{-}cb}(H)$. Take $\alpha$ and $\beta$ like in (\ref{phiab}), so 
	\begin{equation*}
	\varphi(st^{-1})=\langle\beta(s),\alpha(t)\rangle,\quad \forall s,t\in G.
	\end{equation*}
	Then
	\begin{align}
	\varphi\ast\check{\psi}(st^{-1}) &= \int_G\psi(ts^{-1}r)\varphi(r)\,dr \nonumber\\
	&= \int_G\psi(r)\varphi(st^{-1}r)\,dr\nonumber\\
	&= \int_G\psi(r)\langle\beta(s),\alpha(r^{-1}t)\rangle\,dr \nonumber\\
	&= \langle\beta(s),\tilde{\alpha}(t)\rangle,\label{alphatild}
	\end{align}
	where
	\begin{equation*}
	\tilde{\alpha}(t)=\int_G \psi(r)\alpha(r^{-1}t)\,dr.
	\end{equation*}
	Recall that the function $\psi$ has compact support, so $\tilde{\alpha}$ is well defined and continuous, and the equality (\ref{alphatild}) is justified. Moreover,
	\begin{equation*}
	\|\tilde{\alpha}(t)\|\leq\int_G |\psi(r)|\|\alpha(r^{-1}t)\|\,dr \leq \sup_{t\in G}\|\alpha(t)\|,\quad\forall t\in G.
	\end{equation*}
	So, in particular,
	\begin{equation*}
	\left(\sup_{s\in H}\|\beta(t)\|\right)\left(\sup_{t\in H}\|\tilde{\alpha}(t)\|\right) \leq \left(\sup_{s\in G}\|\beta(t)\|\right)\left(\sup_{t\in G}\|\alpha(t)\|\right).
	\end{equation*}
	Thus $\left(\varphi\ast\check{\psi}\right)\!|_H\in M_{p\text{-}cb}(H)$ and $\left\|\left(\varphi\ast\check{\psi}\right)\!|_H\right\|_{M_{p\text{-}cb}(H)}\leq\|\varphi\|_{M_{p\text{-}cb}(G)}$. Using (\ref{Phi(f),phi}) we get
	\begin{equation*}
	\left|\langle\Phi(f),\varphi\rangle\right|\leq \|f\|_{M_{p\text{-}cb}(H)^*}\left\|\left(\varphi\ast\check{\psi}\right)\!|_H\right\|_{M_{p\text{-}cb}(H)}\leq \|f\|_{M_{p\text{-}cb}(H)^*}\|\varphi\|_{M_{p\text{-}cb}(G)}.
	\end{equation*}
	Hence, $\Phi$ extends to a contraction from $Q_{p\text{-}cb}(H)$ to $Q_{p\text{-}cb}(G)$, and its adjoint $\Phi^*:M_{p\text{-}cb}(G)\to M_{p\text{-}cb}(H)$ is given by $\Phi^*(\varphi)=\left(\varphi\ast\check{\psi}\right)\!|_H$. Observe that $\Phi^*(1)=1$, so if we prove that $\Phi^*$ maps $A_p(G)$ to $A_p(H)$, we can conclude that $H$ has the $p\,$-AP as in Proposition \ref{G/KpAP}. Take $f\in L_p(G)$ and $g\in L_{p'}(G)$, where $p'$ is the H\"older conjugate of $p$, and observe that
	\begin{equation*}
	\left(g\ast\check{f}\right)\ast\check{\psi}=g\ast\left(\psi\ast f\right)^{\vee}.
	\end{equation*}
	Since $\|\psi\ast f\|_p\leq\|f\|_p$, we get
	\begin{equation*}
	\left\|\left(\sum_{n=1}^Ng_n\ast\check{f}_n\right)\ast\check{\psi}\right\|_{A_p(G)}\leq \sum_{n=1}^N\|g_n\|_{p'}\|\psi\ast f_n\|_p
	\leq \sum_{n=1}^N\|g_n\|_{p'}\|f_n\|_p,
	\end{equation*}
	for all finite families $f_1,...,f_N\in L_p(G)$, $g_1,...,g_N\in L_{p'}(G)$. By density, this implies that $\|a\ast\check{\psi}\|_{A_p(G)}\leq \|a\|_{A_p(G)}$ for all $a\in A_p(G)$. Since $H$ is a closed subgroup of $G$, then the restriction $\left(a\ast\check{\psi}\right)\!|_H$ is an element of $A_p(H)$ and $\left\|\left(a\ast\check{\psi}\right)\!|_H\right\|_{A_p(H)}\leq\|a\ast\check{\psi}\|_{A_p(G)}$ (see e.g. \cite[Theorem 7.8.2]{Der}). Therefore, $\Phi^*$ maps $A_p(G)$ to $A_p(H)$. Thus, if $(\varphi_i)$ is a net in $A_p(G)$ such that $\varphi_i\to 1$ in $\sigma(M_{p\text{-}cb}(G),Q_{p\text{-}cb}(G))$, then $(\Phi^*(\varphi_i))$ is a net in $A_p(H)$ such that $\Phi^*(\varphi_i)\to 1$ in $\sigma(M_{p\text{-}cb}(H),Q_{p\text{-}cb}(H))$.
\end{proof}

\begin{rmk}
	For $p=2$, Proposition \ref{HpAP} corresponds to \cite[Proposition 1.14]{HaaKra}, whose proof is very different. We are not able to adapt that proof since there is no known analogue of \cite[Theorem 1.11]{HaaKra} for $p\neq 2$. This is related to the fact that, as observed by Daws, we do not have a simple description of what $PM_p(\text{SU}(2))$ is, unless $p=2$ (see the remark right after \cite[Proposition 8.7]{Daw}).
\end{rmk}

Now let $\Gamma$ be lattice in $G$. This means that $\Gamma$ is a discrete subgroup of $G$ such that $G/\Gamma$ has a finite $G$-invariant measure. This implies the existence of a Borel subset $\Omega\subset G$ such that the restriction of the quotient map $G\to G/\Gamma$ to $\Omega$ is bijective. Observe that this allows to define maps $\omega:G\to\Omega$, $\gamma:G\to\Gamma$ such that
\begin{equation}\label{s=omga}
s=\omega(s)\gamma(s),\quad\forall s\in G,
\end{equation}
and this decomposition is unique. Let $\mu_G$ be the Haar measure on $G$. Since $\Gamma$ is a lattice, $\mu_G(\Omega)$ is finite and strictly positive, so we may normalize $\mu_G$ in such a way that $\mu_G(\Omega)=1$. The following proposition corresponds to \cite[Theorem 2.4]{HaaKra} for $p=2$ and the proof uses the same ideas.

\begin{prop}\label{pAPGamma}
	Let $1<p<\infty$. If $\Gamma$ has the $p\,$-AP, then so does $G$.
\end{prop}
\begin{proof}
	Again the strategy is to define a suitable map $\Phi: M_{p\text{-}cb}(\Gamma)\to M_{p\text{-}cb}(G)$ and consider the image of an approximating net in $A_p(\Gamma)$. For $\varphi\in M_{p\text{-}cb}(\Gamma)$, define $\tilde{\Phi}(\varphi)\in L_\infty(G)$ by
	\begin{align*}
	\tilde{\Phi}(\varphi) (s) = \varphi(\gamma(s)),\qquad\forall\, s\in G,
	\end{align*}
	and put $\Phi(\varphi)=\tilde{\Phi}(\varphi)\ast\check{\mathds{1}}_\Omega$, where $\check{\mathds{1}}_\Omega$ stands for the indicator function of the set $\Omega^{-1}$.
	Thus,
	\begin{equation*}
	\Phi(\varphi)(s)=\int_{\Omega}\varphi(\gamma(su))\,du,\quad\forall s\in G.
	\end{equation*}
	This map satisfies $\Phi(1)=1$. Let us show that $\Phi(\varphi)\in M_{p\text{-}cb}(G)$. Take $E$, $\alpha$ and $\beta$ as in (\ref{phiab}). We have
	\begin{equation*}
	\Phi(\varphi)(st^{-1})=\int_{\Omega}\varphi(\gamma(st^{-1}u))\,du,\quad\forall s,t\in G.
	\end{equation*}
	As was shown in the proof of \cite[Lemma 2.1]{Haa}, the map $u\in\Omega\mapsto\omega(tu)\in\Omega$ preserves $\mu_G$ for all $t\in G$. Hence we get
	\begin{equation*}
	\Phi(\varphi)(st^{-1})=\int_{\Omega}\varphi(\gamma(st^{-1}\omega(tu)))\,du,\quad\forall s,t\in G.
	\end{equation*}
	Now observe that
	\begin{equation*}
	st^{-1}\omega(tu)=su(tu)^{-1}\omega(tu)=\omega(su)\gamma(su)\gamma(tu)^{-1}.
	\end{equation*}
	Since the decomposition (\ref{s=omga}) is unique, we get $\gamma(st^{-1}\omega(tu))=\gamma(su)\gamma(tu)^{-1}$. Thus
	\begin{equation*}
	\Phi(\varphi)(st^{-1})=\int_{\Omega}\varphi(\gamma(su)\gamma(tu)^{-1})\,du=\int_{\Omega}\langle\beta(\gamma(su)),\alpha(\gamma(tu))\rangle\,du,\quad\forall s,t\in G.
	\end{equation*}
	Therefore, if we define $\tilde{\alpha}:G\to L_p(\Omega;E)$ by $\tilde{\alpha}(t)=\alpha(\gamma(t\ \cdot))$, and $\tilde{\beta}:G\to L_p(\Omega;E)^*$ by
	\begin{equation*}
	\langle\tilde{\beta}(s),f\rangle = \int_\Omega \langle\beta(\gamma(su)),f(u)\rangle\,du,\quad \forall s\in G,\ \forall f\in L_p(\Omega;E),
	\end{equation*}
	we get $\Phi(\varphi)(st^{-1})=\langle\tilde{\beta}(s),\tilde{\alpha}(t)\rangle$ for all $s,t\in G$. Moreover,
	\begin{equation}\label{normalphtil}
	\|\tilde{\alpha}(t)\|_{L_p(\Omega;E)}=\left( \int_\Omega \|\alpha(tu)\|^p\, du\right)^{\frac{1}{p}}\leq\|\alpha\|_{\infty},\quad \forall t\in G.
	\end{equation}
	And putting $p'=\frac{p}{p-1}$, for all $s\in G$ and $f\in L_p(\Omega;E)$,
	\begin{align}
	|\langle\tilde{\beta}(s),f\rangle | & \leq \int_\Omega \|\beta(\gamma(su))\| \|f(u)\|\,du\nonumber \\
	& \leq \left(\int_\Omega \|\beta(\gamma(su))\|^{p'}\,du\right)^{\frac{1}{p'}} \left(\int_\Omega  \|f(u)\|^p\,du\right)^{\frac{1}{p}}\nonumber\\
	& \leq \|\beta\|_{\infty}  \|f\|_{L_p(\Omega;E)}.\label{normbettil}
	\end{align}
	Hence, $\|\tilde{\alpha}\|_{\infty}\leq\|\alpha\|_{\infty}$ and $\|\tilde{\beta}\|_{\infty}\leq\|\beta\|_{\infty}$. Observe that $\Phi(\varphi)$ is continuous because it is the convolution of $\tilde{\Phi}(\varphi)\in L_\infty(G)$ and $\check{\mathds{1}}_\Omega\in L_1(G)$ (see e.g. \cite[Proposition 2.39]{Foll}). In terms of \cite[Theorem 8.6]{Daw}, this, together with Remark \ref{SQpLp}, is enough to conclude that $\Phi(\varphi)\in M_{p\text{-}cb}(G)$ and $\|\Phi(\varphi)\|_{M_{p\text{-}cb}(G)}\leq\|\alpha\|_{\infty}\|\beta\|_{\infty}$. Taking the infimum over all the decompositions, we obtain
	\begin{equation*}
	\|\Phi(\varphi)\|_{M_{p\text{-}cb}(G)}\leq\|\varphi\|_{M_{p\text{-}cb}(\Gamma)}.
	\end{equation*}
	Let us show now that $\Phi$ is an adjoint operator. Let $\varphi\in M_{p\text{-}cb}(\Gamma)$ and $f\in L_1(G)$. We have
	\begin{align*}
	\langle f,\Phi(\varphi)\rangle &= \int_G f(s)\int_\Omega\varphi(\gamma(su))\, du\, ds \\
	&= \int_\Omega\int_G f(s)\varphi(\gamma(su))\, ds\, du\\
	&= \int_\Omega\int_G f(su^{-1})\varphi(\gamma(s))\, ds\, du \\
	&= \int_\Omega \left(\sum_{r\in\Gamma}\int_{\Omega r} f(su^{-1})\varphi(\gamma(s))\, ds\right)\, du\\
	&= \sum_{r\in\Gamma}\varphi(r)\int_\Omega \int_{\Omega r} f(su^{-1})\, ds\, du.
	\end{align*}
	In the first line we have used Fubini's theorem, which holds because
	\begin{equation*}
	\int_\Omega\int_G |f(s)| |\varphi(\gamma(su))|\, ds\, du \leq \|f\|_1 \|\varphi\|_\infty.
	\end{equation*}
	Thus, defining
	\begin{equation*}
	\Psi(f)(r)=\int_\Omega \int_{\Omega r} f(su^{-1})\, ds\, du,\quad\forall f\in L_1(G),\ \forall r\in\Gamma,
	\end{equation*}
	the previous computations show that $\Psi(f)\in\ell_1(\Gamma)$ and that $\Phi=\Psi^*$, as in the proof of Proposition \ref{G/KpAP}. We conclude that $\Phi:M_{p\text{-}cb}(\Gamma)\to M_{p\text{-}cb}(G)$ is weak*-weak* continuous. Finally, we need to check that $\Phi$ maps $A_p(\Gamma)$ to $A_p(G)$. Take $f\in \ell_p(\Gamma)$ and $g\in\ell_{p'}(\Gamma)$, and define
	\begin{equation*}
	\tilde{f}(s)=f(\gamma(s)),\quad \tilde{g}(s)=g(\gamma(s)),\quad\forall s\in G.
	\end{equation*}
	Observe that
	\begin{equation*}
	\int_G|\tilde{f}(s)|^p\,ds=\sum_{r\in\Gamma}\int_{\Omega r}|f(r)|^p\,ds=\|f\|_p^p.
	\end{equation*}
	This shows that $\tilde{f}\in L_p(G)$ and $\|\tilde{f}\|_p\leq\|f\|_p$. Analogously $\|\tilde{g}\|_{p'}\leq\|g\|_{p'}$. Moreover, for every $s\in G$,
	\begin{align*}
	\tilde{g}\ast\check{\tilde{f}}(s) &= \tilde{f}\ast\check{\tilde{g}}(s^{-1}) \\
	&=\int_G \tilde{f}(t)\tilde{g}(st)\,dt \\
	&= \int_G f(\gamma(t))g(\gamma(st))\,dt\\
	&=\sum_{r\in\Gamma}f(r)\int_{\Omega r}g(\gamma(st))\,dt \\
	&=\sum_{r\in\Gamma}f(r)\int_{\Omega}g(\gamma(str))\,dt\\
	&=\int_{\Omega}\sum_{r\in\Gamma}f(r)g(\gamma(st)r)\,dt \\
	&=\int_{\Omega}f\ast\check{g}(\gamma(st)^{-1})\,dt\\
	&=\int_{\Omega}g\ast\check{f}(\gamma(st))\,dt \\
	&= \Phi(g\ast\check{f})(s).
	\end{align*}
	Hence $\Phi(g\ast\check{f})\in A_p(G)$ and $\|\Phi(g\ast\check{f})\|_{A_p(G)}\leq\|g\|_{p'}\|f\|_p$. By linearity we obtain
	\begin{equation*}
	\left\|\Phi\left(\sum g_n\ast\check{f}_n\right)\right\|_{A_p(G)}\leq\sum\|g_n\|_{p'}\|f_n\|_p,
	\end{equation*}
	for all finite families $(f_n)$ in $\ell_p(\Gamma)$ and $(g_n)$ in $\ell_{p'}(\Gamma)$. Taking the infimum over all the decompositions we get $\|\Phi(a)\|_{A_p(G)}\leq\|a\|_{A_p(\Gamma)}$ for all $a\in \Lambda_p(\ell_{p'}(\Gamma)\otimes \ell_p(\Gamma))\subset A_p(\Gamma)$, where $\Lambda_p$ is the map defined in (\ref{Lamda_p}). By density the same holds for all $a\in A_p(\Gamma)$. Thus, if $(\varphi_i)$ is a net in $A_p(\Gamma)$ such that $\varphi_i\to 1$ in $\sigma(M_{p\text{-}cb}(\Gamma),Q_{p\text{-}cb}(\Gamma))$, then $(\Phi(\varphi_i))$ is a net in $A_p(G)$ such that $\Phi(\varphi_i)\to 1$ in $\sigma(M_{p\text{-}cb}(G),Q_{p\text{-}cb}(G))$.
\end{proof}

\subsection{Proof of the theorem}

Now we are ready to give the proof of Theorem \ref{LiepAP}, which relies on the following lemma.

\begin{lem}\label{G1G2}
	Let $1<p<\infty$ and $G_1, G_2$ be two locally isomorphic connected simple Lie groups with finite center. Then $G_1$ has the $p\,$-AP if and only if $G_2$ has the $p\,$-AP.
\end{lem}
\begin{proof}
	Since the groups are locally isomorphic, their Lie algebras are isomorphic, which in turn implies that their adjoint groups are isomorphic (see \cite[\S II.5]{Hel} for more details). If we denote the center of $G_i$ by $Z(G_i)$, then by \cite[Corollary II.5.2]{Hel}, $G_1/Z(G_1)$ and $G_2/Z(G_2)$ are isomorphic. Then, using Proposition \ref{G/KpAP}, we see that $G_i$ has the $p\,$-AP if and only if $G_i/Z(G_i)$ does. Since the quotients are isomorphic, the result follows.
\end{proof}

\begin{proof}[Proof of Theorem \ref{LiepAP}]
	First let us prove part (\textit{a}). If $G$ has real rank greater than 1, then by \cite[Proposition I.1.6.2]{Mar}, it has a subgroup $H$ which is locally isomorphic to either $\text{SL}(3,\R)$ or $\text{Sp}(2,\R)$. Moreover, as was shown in \cite[\S 4]{Dor}, $H$ is a closed subgroup of $G$. By Lemma \ref{G1G2}, $H$ cannot have the $p\,$-AP for any $1<p<\infty$, or else it would contradict Theorem \ref{SL3pAP} or Theorem \ref{Sp2pAP}. Furthermore, by Proposition \ref{HpAP}, the same holds for $G$. On the other hand, it was shown in \cite{CowHaa} that if $G$ has real rank 1, it is weakly amenable. And it follows from the $KAK$ decomposition that if $G$ has real rank 0, it is compact, which in turns implies that it is amenable. By \cite[Theorem 1.12]{HaaKra}, these properties are stronger than the AP, so we conclude using Proposition \ref{q-APp-AP}. Part (\textit{b}) is a consequence of (\textit{a}) plus Propositions \ref{HpAP} and \ref{pAPGamma}.
\end{proof}

\section{The $p\,$-AP implies that the convoluters are pseudo-measures}

Let $G$ be a locally compact group. It was proven in \cite{DawSp} that the approximation property AP implies that $CV_p(G)=PM_p(G)$ for all $1<p<\infty$. In this section we show how their arguments actually prove the following somewhat stronger result.

\begin{thm}\label{CV_p=PM_p}
	Let $1<p<\infty$. If $G$ has the $p\,$-AP, then $CV_p(G)=PM_p(G)$.
\end{thm}

\begin{rmk}
	By Proposition \ref{q-APp-AP}, the previous theorem implies that, if $G$ has the $q\,$-AP, then $CV_p(G)=PM_p(G)$ for all $1<p,q<\infty$ such that $\left|\frac{1}{q}-\frac{1}{2}\right|\leq\left|\frac{1}{p}-\frac{1}{2}\right|$.
\end{rmk}

Recall that $PM_p(G)$ is the dual of $A_p(G)$. Cowling showed  \cite{Cow} that $CV_p(G)$ can also be viewed as a dual space. Let $e$ be the identity element of $G$. For every compact neighbourhood $K$ of $e$ in $G$, let $L_p(K)$ be the subspace of $L_p(G)$ consisting of those functions supported on $K$, and let $\check{A}_{p,K}(G)$ be the space of functions of the form
\begin{equation*}
a=\sum_n g_n\ast \check{f}_n,\quad g_n\in L_{p'}(K),\ f_n\in L_p(K),
\end{equation*}
such that $\displaystyle\sum_n\|g_n\|_{p'}\|f_n\|_p<\infty$. Let $\check{A}_{p}(G)=\displaystyle\bigcup_K\check{A}_{p,K}(G)$ and
\begin{align*}
\|a\|_{\check{A}_{p}}=\inf\sum_n\|g_n\|_{p'}\|f_n\|_p,
\end{align*}
where the infimum ranges over all the decompositions $a=\sum g_n\ast \check{f}_n\in \check{A}_{p,K}(G)$ and all the compact neighbourhoods $K$ of $e$. Then $\check{A}_{p}(G)^*$ may be identified with $CV_p(G)$ by $T\in CV_p(G)\mapsto \Phi_T\in\check{A}_{p}(G)^*$, where
\begin{equation*}
\langle\Phi_T,a\rangle=\sum\langle g_n,T(f_n)\rangle,\quad a=\sum g_n\ast \check{f}_n.
\end{equation*}
For every $\tau\in L_{p'}(G)\hat{\otimes}L_p(G)$, $T\in CV_p(G)$ and $\psi\in\check{A}_{p}(G)$ such that $\psi\geq 0$ and $\int \psi=1$, we can define $\mu\in M_{p\text{-}cb}(G)^*$ by
\begin{equation*}
\langle\mu,\varphi\rangle=\langle T, (\psi \ast \varphi)\cdot\tau\rangle,\quad\forall\varphi\in M_{p\text{-}cb}(G),
\end{equation*}
using the duality $\mathcal{B}(L_p(G))=(L_{p'}(G)\hat{\otimes}L_p(G))^*$. Here we view $\tau$ as a function on $G\times G$ and
\begin{equation*}
\varphi\cdot\tau (s,t)=\varphi(st^{-1})\tau(s,t),\quad \forall\varphi\in M_{p\text{-}cb}(G),\ \forall\tau\in L_{p'}(G)\hat{\otimes}L_p(G).
\end{equation*}
Under these conditions, $\|\mu\|\leq\|\tau\|\|T\|$.
\begin{prop}
	Let $\mu$ be as above. Then $\mu$ is $\ast$-weak continuous, that is, $\mu\in Q_{p\text{-}cb}(G)$.
\end{prop}
\begin{proof}
	Since $C_c(G)\otimes C_c(G)$ is dense in $L_{p'}(G)\hat{\otimes}L_p(G)$, by continuity we may assume $\tau\in C_c(G)\otimes C_c(G)$. We will construct $g\in L_1(G)$ such that $\langle\mu,\varphi\rangle=\int \varphi g$ for all $\varphi\in M_{p\text{-}cb}(G)$. Recall the definition of $\Lambda_p$ in (\ref{Lamda_p}) and observe that $a=\Lambda_p(\tau)\in\check{A}_{p}(G)$. We have
	\begin{equation*}
	(\psi\ast \varphi)(s)a(s)=(\psi\ast \chi_S\varphi)(s)a(s),\quad\forall s\in G,
	\end{equation*}
	where $S=\text{supp}(\psi)^{-1}\text{supp}(a)$ is a compact subset of $G$. Define $\psi_t(s)=\psi(st)$ and
	\begin{equation*}
	g(t)=\chi_S(t)\Delta(t^{-1})\langle T,\psi_{t^{-1}}a\rangle,\quad t\in G.
	\end{equation*}
	Then $g\in L_1(G)$ and
	\begin{equation*}
	\int_G \varphi(t)g(t)\,dt=\langle\mu,\varphi\rangle,\quad\forall \varphi\in M_{p\text{-}cb}(G).
	\end{equation*}
	See \cite[Lemma 3.2]{DawSp} for details.
\end{proof}

\begin{proof}[Proof of Theorem \ref{CV_p=PM_p}]
It is always true that $PM_p(G)\subseteq CV_p(G)$. Now let $T\in CV_p(G)$ and $(\varphi_i)$ be a net in $A_{p,c}(G)$ such that $\varphi_i\to 1$ in $\sigma(M_{p\text{-}cb}(G),Q_{p\text{-}cb}(G))$. Fix $\psi\in\check{A}_{p}(G)$ as above and take $\tau\in L_{p'}(G)\hat{\otimes}L_p(G)$. By the previous proposition, there exists $\mu\in Q_{p\text{-}cb}(G)$ such that $\langle \varphi,\mu\rangle=\langle T, (\psi \ast \varphi)\cdot\tau\rangle$ for all $\varphi\in M_{p\text{-}cb}(G)$. Therefore
	\begin{align*}
	\lim_i\langle(\psi\ast \varphi_i)\cdot T,\tau\rangle &= \lim_i\langle T,(\psi\ast \varphi_i)\cdot\tau\rangle \\
	&= \lim_i\langle \varphi_i,\mu\rangle \\
	&= \langle 1,\mu\rangle\\
	&= \langle T,(\psi\ast 1)\cdot\tau\rangle \\
	&= \langle T,1\cdot\tau\rangle \\
	&= \langle T,\tau\rangle.
	\end{align*}
	Here we have used the $M_{p\text{-}cb}(G)$-module structure of $CV_p(G)$ (see \cite[Corollary 2.7]{DawSp}). This is valid for all $\tau\in L_{p'}(G)\hat{\otimes}L_p(G)$, so $(\psi\ast \varphi_i)\cdot T\to T$ in the weak* topology of $\mathcal{B}(L_p(G))$. Since $\psi$ and $\varphi_i$ have compact support, so does $\psi\ast \varphi_i$. Thus, by \cite[Corollary 2.8]{DawSp}, $(\psi\ast \varphi_i)\cdot T\in PM_p(G)$ for all $i$, and since $PM_p(G)$ is weak* closed, this implies that $T\in PM_p(G)$.
\end{proof}

\section*{Acknowledgements}
I would like to thank Mikael de la Salle for many enlightening discussions and for his careful reading of some previous versions of this paper. I also thank the two anonymous referees for their valuable comments and suggestions. This work was partially supported by CONICYT - Becas Chile.

\bibliographystyle{plain} 
\bibliography{Bibliography}

\end{document}